\documentclass[11pt, a4paper]{article}

\usepackage{latexsym}
\usepackage{graphicx, float}
\usepackage{amssymb,amsfonts,amsthm,amsmath,graphicx,url}

\newcommand{\SM}[1]{{\bf [Zhou, 2011: #1]}}
\long\def\delete#1{}

\newtheorem{theorem}{Theorem}
\newtheorem{lemma}[theorem]{Lemma}
\newtheorem{definition}[theorem]{Definition}
\newtheorem{example}[theorem]{Example}

\newtheorem{remark}[theorem]{Remark}

\input amssym.def
\input amssym.tex

\newcommand{\bmat}[1]{\begin{bmatrix}#1\end{bmatrix}}
\newcommand{\dmat}[1]{\begin{vmatrix}#1\end{vmatrix}}

\newcommand{\be}{\begin{equation}}
\newcommand{\ee}{\end{equation}}
\newcommand{\bea}{\begin{eqnarray}}
\newcommand{\eea}{\end{eqnarray}}
\newcommand{\bean}{\begin{eqnarray*}}
\newcommand{\eean}{\end{eqnarray*}}
\def\non{\nonumber}

\def\qed{\hfill$\Box$\vspace{12pt}}

\def\la{\langle}
\def\ra{\rangle}

\def\FFF{\Bbb F}

\def\FFF{\mathbb{F}}

\def\BB{{\cal B}}
\def\CC{{\cal C}}
\def\DD{{\cal D}}

\def\b0{{\bf 0}}

\def\bc{{\bf c}}

\def\be{{\bf e}}

\def\bu{{\bf u}}
\def\bv{{\bf v}}

\def\Ga{\Gamma}

\def\Si{\Sigma}
\def\Om{\Omega}

\def\a{\alpha}
\def\b{\beta}

\def\l{\lambda}
\def\s{\sigma}
\def\t{\tau}

\def\Aut{{\rm Aut}}

\def\PSL{{\rm PSL}}

\def\PSU{{\rm PSU}}
\def\PGU{{\rm PGU}}
\def\SL{{\rm SL}}
\def\GL{{\rm GL}}
\def\GU{{\rm GU}}
\def\SU{{\rm SU}}

\def\AG{{\rm AG}}
\def\PG{{\rm PG}}

\def\AGammaL{{\rm A\Gamma L}}
\def\PGammaL{{\rm P\Gamma L}}
\def\PGammaU{{\rm P\Gamma U}}

\def\Sp{{\rm Sp}}

\title{Unitary graphs and classification of a family of symmetric graphs with complete quotients}
\author{Massimo Giulietti$^1$, Stefano Marcugini$^1$, Fernanda Pambianco$^1$\\ and Sanming Zhou$^2$\\
$^1$Dipartimento di Matematica e Informatica\\ 
Universit\`a di Perugia, 06128 Perugia, Italy\\ 
Emails:  \{giuliet, gino, fernanda\}@dipmat.unipg.it \\
\\
$^2$Department of Mathematics and Statistics\\
The University of Melbourne, VIC 3010, Australia\\
Email: smzhou@ms.unimelb.edu.au\\}

\date{18/12/2012}
\begin{document}
\openup 0.8\jot
\maketitle


\begin{abstract}
A finite graph $\Ga$ is called $G$-symmetric if $G$ is a group of automorphisms of $\Ga$ which is transitive on the set of ordered pairs of adjacent vertices of $\Ga$. We study a family of symmetric graphs, called the unitary graphs, whose vertices are flags of the Hermitian unital and whose adjacency relations are determined by certain elements of the underlying finite fields. Such graphs admit the unitary groups as groups of automorphisms, and they play a significant role in the classification of a family of symmetric graphs with complete quotients such that an associated incidence structure is a doubly point-transitive linear space. We give this classification in the paper and also investigate combinatorial properties of the unitary graphs.   

\end{abstract}

\section{Introduction}

This paper was motivated by our interest of classifying a family of symmetric graphs with complete quotients such that a certain design involved is a doubly point-transitive linear space. It is known that for such a linear space the group involved is almost simple or contains a regular normal subgroup which is elementary abelian. We handle the almost simple case in this paper. It turns out that the most interesting graphs arisen from this classification are what we call the unitary graphs. Their vertices are the flags of the Hermitian unital and their adjacency relations are determined by certain elements of $\FFF_{q^2}^*$ (see the next section for precise definition). The groups involved in the unitary graphs are the unitary groups between $\PGU(3,q)$ and $\PGammaU(3, q)$.  

Let $G$ be a finite group and $\Ga$ a finite graph with vertex set $V(\Ga)$. Suppose $G$ acts on $V(\Ga)$ as a group of automorphisms of $\Ga$, that is, $G$ preserves the adjacency and non-adjacnecy relations of $\Ga$. If $G$ is transitive on $V(\Ga)$ and, in its induced action, transitive on the set of arcs of $\Ga$, then $\Ga$ is said to be {\em $G$-symmetric}, where an {\em arc} is an ordered pair of adjacent vertices. A $G$-symmetric graph is also called {\em $G$-arc transitive} in the literature. There is an extensive literature on symmetric and highly arc-transitive graphs beginning with \cite{Tutte}. The reader is referred to two useful surveys \cite{Praeger97, Praeger00} in this area. 

For a $G$-symmetric graph $\Ga$, if $V(\Ga)$ admits a nontrivial $G$-invariant partition $\BB = \{B, C, \ldots\}$, that is, $1 < |B| < |V(\Ga)|$ and any element of $G$ maps blocks of $\BB$ to blocks of $\BB$, then we call $\Ga$ an imprimitive $G$-symmetric graph. In this case the {\em quotient graph} $\Ga_{\BB}$ of $\Ga$ relative to $\BB$ is
defined to be the graph with vertex set $\BB$ in which $B, C \in
\BB$ are adjacent if and only if there exists at least one edge of $\Ga$ between $B$ and $C$. We assume without explicit mentioning that $\Ga_{\BB}$ has at least one edge. Since $\Ga$ is $G$-symmetric and $\BB$ is $G$-invariant, this implies that each block of $\BB$ is an independent set of $\Ga$. Denote by $\Ga(\a)$ the neighbourhood of $\a \in V(\Ga)$ in $\Ga$ and set $\Ga(B) = \cup_{\a \in
B}\Ga(\a)$. For $C \in \BB$ adjacent to $B$ in $\Ga_{\BB}$, we call $m = |\{D \in \BB: \Ga(D) \cap B = \Ga(C) \cap B\}|$ the {\em multiplicity} of $\BB$. Since $\Ga$ is $G$-symmetric and $\BB$ is $G$-invariant, $|B|$, $|\Ga(C) \cap B|$ and $m$ are all independent of the choice of $B$ and $C$. If $|\Ga(C) \cap B|=|B|$ or $|\Ga(C) \cap B|=|B|-1$, then $\Ga$ is called a {\em multicover} (e.g.~\cite{LPVZ}) or {\em almost multicover} of $\Ga_{\BB}$ respectively; if in addition the edges between $B$ and $C$ form a matching, then $\Ga$ is called a {\em cover} or {\em almost cover} \cite{Zhou98} of $\Ga_{\BB}$ respectively.     

A natural incidence structure $\DD(\Ga, \BB)$ \cite{Zhou-EJC} arises when $\Ga$ is an almost multicover of $\Ga_{\BB}$. Its points are the blocks of $\BB$ and its blocks are the images of $\BB(\a) \cup \{B\}$ under the action of $G$, where $\a \in B$ is fixed and $\BB(\a) = \{C \in \BB: \Ga(C) \cap B = B \setminus \{\a\}\}$. The incidence relation of $\DD(\Ga, \BB)$ is the set-theoretic inclusion. In general, $\DD(\Ga, \BB)$ is a 1-design of block size $m+1$ \cite[Lemma 2.2]{Zhou-EJC}. In the special case when $\Ga_{\BB}$ is a complete graph, $\DD(\Ga, \BB)$ is a $2$-design that admits $G$ as a doubly point-transitive and block-transitive group of automorphisms. A program set up in \cite{Zhou-EJC} is to classify all possible $\Ga$ in the special case when $\Ga_{\BB}$ is complete and this 2-design is a linear space. (A {\em linear space} \cite{Beth-Jung-Lenz} is an incidence structure of points and lines such that any point is incident with at least two lines, any line with at least two points, and any two points are incident with exactly one line.) When this linear space is trivial (that is, each line is incident with exactly two points), all graphs are classified in \cite[Theorem 3.19]{Zhou-EJC} and interesting graphs arise, including the cross ratio graphs \cite{Gardiner-Praeger-Zhou99, Zhou-EJC} from finite projective lines. 

All nontrivial doubly point-transitive linear spaces are known \cite{Kantor85} and the corresponding group $G$ is {\em almost simple} (that is, $G$ has a nonabelian simple normal subgroup $N$ such that $N \unlhd G \le \Aut(N)$) or contains a regular normal subgroup which is elementary abelian. In this paper we classify $(\Ga, G, \BB)$ such that $\Ga_{\BB}$ is complete and almost multi-covered by $\Ga$, $\DD(\Ga, \BB)$ is a nontrivial linear space and $G$ is almost simple. The most interesting graphs arisen from this classification are the unitary graphs; see Definition \ref{def:ugraph}. Let $\Ga^+(P; d, q)$ ($\Ga^{\simeq}(P; d, q)$, respectively) be the graph \cite{Zhou-EJC} with vertices the point-line flags of $\PG(d-1, q)$ such that two such flags $(\s, L), (\t, N)$ are adjacent if and only if $L, N$ are intersecting (skew, respectively) in $\PG(d-1, q)$. The following is the main result in this paper. 

\begin{theorem}
\label{thm:class}
Suppose $\Ga$ is a $G$-symmetric graph admitting a nontrivial $G$-invariant 
partition $\BB$ of block size at least 3 such that $\Ga_{\BB}$ is a complete graph, $\Ga$ is an almost multicover of $\Ga_{\BB}$ and $\DD(\Ga, \BB)$ is a nontrivial linear space. Suppose further that $G$ is almost simple. Then one of the following occurs:
\begin{itemize}
\item[\rm (a)] $\Ga$ is isomorphic to $\Ga^+(P; d, q)$ or $\Ga^{\simeq}(P; d, q)$, and $\PSL(d, q) \unlhd G \le \PGammaL(d, q)$, for some integer $d \ge 3$ and prime power $q$;
\item[\rm (b)] $\Ga$ is isomorphic to a unitary graph $\Ga_{r,\l}(q)$ and $\PGU(3, q) \unlhd G \le \PGammaU(3, q)$, for a prime power $q > 2$ and appropriate $r, \l$;
\item[\rm (c)] $\Ga$ is isomorphic to one of the four graphs whose vertices are the flags of $\PG(3,2)$, and $G = A_7$; these graphs have order $105$ and (valency, diameter, girth) = $(24, 2, 3)$, $(24, 3, 3)$, $(12, 3, 4)$, $(12, 3, 3)$ respectively.
\end{itemize}
\end{theorem}

The unitary graphs $\Ga_{r,\l}(q)$ in (b) will be defined in Definition \ref{def:ugraph}; they are the main objects of study in this paper. The four graphs in (c) will be described in the proof of Theorem \ref{thm:class} in Section \ref{sec:class'n}.

Theorem \ref{thm:class} relies on the classification \cite{Kantor85} of doubly point-transitive linear spaces (which relies on the classification of finite simple groups) and the flag graph construction introduced in \cite{Zhou-EJC}. A major part of the proof of Theorem \ref{thm:class} is to analyze the unitary graphs. This will be carried out in Section \ref{sec:char}. As we will see later, the vertex set of a unitary graph admits two natural partitions such that one of the corresponding quotient graphs is a complete graph and the other one is not. It seems that the second quotient is interesting, and we will study its combinatorial properties in Section \ref{sec:2nd}. 

In a recent paper \cite{Zhou12} the fourth-named author used unitary graphs to obtain a lower bound on the largest number of vertices in a symmetric graph with diameter two and degree $q(q^2 - 1)$ for a prime power $q > 2$.  

The reader is referred to \cite{Dixon-Mortimer} and \cite{Beth-Jung-Lenz} for undefined terminology on permutation groups and combinatorial designs respectively. This paper forms part of the fourth author's project of studying imprimitive symmetric graphs; see \cite{Gardiner-Praeger95, IPZ, Li-Praeger-Zhou98, LZ} and \cite{Zhou02}-\cite{Zhou-EJC} for recent progress in this direction.

\section{Definition of the unitary graphs}
\label{sec:ugraph}

Before giving the definition of the unitary graphs, we gather basic results on the unitary groups and the Hermitian unitals. The reader is referred to \cite{O'Nan, Taylor}, \cite[Appendix A]{HKT} and \cite[Section II.8]{HP} for more details. 

Let $q = p^e > 2$ with $p$ a prime. The mapping $\s: x \mapsto x^q$ is an automorphism of the Galois field $\mathbb{F}_{q^2}$, which we will write as $x^q = \bar{x}$ occasionally. The Galois field $\mathbb{F}_{q}$ is then the fixed field of this automorphism. Let $V(3, q^2)$ be a 3-dimensional vector space over $\mathbb{F}_{q^2}$ and $\b: V(3, q^2) \times V(3, q^2) \rightarrow \mathbb{F}_{q^2}$ a nondegenerate $\s$-Hermitian form (that is, $\b$ is sesquilinear such that $\b(a \bu, b \bv) = a b^q \b(\bu, \bv)$ and $\b(\bu, \bv) = \b(\bv, \bu)^q$). The full unitary group $\Gamma U(3, q)$ consists of those semilinear transformations of $V(3, q^2)$ that induce a collineation of $\PG(2, q^2)$ which commutes with $\b$. The general unitary group $\GU(3, q) = \Gamma U(3, q) \cap GL(3, q^2)$ is the group of nonsingular linear transformations of $V(3, q^2)$ leaving $\b$ invariant. The projective unitary group $\PGU(3, q)$ is the quotient group $\GU(3, q)/Z$, where $Z = \{a I: a \in \mathbb{F}_{q^2}, a^{q+1} = 1\}$ is the center of $\GU(3,q)$ and $I$ the identity transformation. The special projective unitary group $\PSU(3, q)$ is the quotient group $\SU(3, q)/(Z \cap \SU(3, q))$, where $\SU(3, q)$ is the subgroup of $\GU(3, q)$ consisting of linear transformations of unit determinant. $\PSU(3, q)$ is equal to $\PGU(3, q)$ if $3$ is not a divisor of $q+1$, and is a subgroup of $\PGU(3, q)$ of index 3 otherwise. It is well known that the automorphism group of $\PSU(3, q)$ is equal to $\PGammaU(3, q) := \PGU(3, q) \rtimes \la \psi \ra$, where $$\psi: x \mapsto x^p, x \in \FFF_{q^2}$$ is the Frobenius map. (We may also view $\psi$ as the element $x' = x^{p}, y' = y^p, z' = z^p$ of $\PGammaU(3,q)$ induced by this map.)

Choosing an appropriate basis for $V(3, q^2)$ allows us to identify vectors of $V(3, q^2)$ with their coordinates and express the corresponding Hermitian matrix of $\b$ by 
$$
D = \bmat{-1 & 0 & 0\\0 & 0 & 1\\0 & 1 & 0}. 
$$
Thus, for $\bu_1 = (x_1, y_1, z_1), \bu_2 = (x_2, y_2, z_2) \in V(3, q^2)$,
$$
\b(\bu_1, \bu_2) = -x_1 x_2^q + y_1 z_2^q + z_1 y_2^q.
$$ 
If $\b(\bu_1, \bu_2) = 0$, then $\bu_1$ and $\bu_2$ are called {\em orthogonal} (with respect to $\b$). A vector $\bu \in V(3, q^2)$ is called {\em isotropic} if it is orthogonal to itself and {\em nonisotropic} otherwise.
Let 
$$
X = \{\la x, y, z \ra: x, y, z \in \FFF_{q^2}, x^{q+1} = y z^q + z y^q\}
$$
be the set of 1-dimensional subspaces of $V(3, q^2)$ spanned by its isotropic vectors. Hereinafter $\la \bu \ra = \la x, y, z \ra$ denotes the 1-dimensional subspace of $V(3, q^2)$ spanned by $\bu = (x, y, z) \in V(3, q^2)$. 
The elements of $X$ are called the {\em absolute points}. It is well known that $|X| = q^3 + 1$, $\PSU(3, q)$ is doubly transitive on $X$, and $\PGammaU(3, q)$ leaves $X$ invariant. Denote 
$$\infty = \la 0, 1, 0 \ra,\;\, 0 = \la 0, 0, 1 \ra.$$
Then $\infty, 0 \in X$. Any point of $X$ other than $\infty$ is of the form $\la x, y, 1 \ra$ and can be viewed as the point $(x, y)$ of $\AG(2, q^2)$ satisfying $x^{q+1} = y + y^q$. 

If $\bu_1$ and $\bu_2$ are isotropic, then the vector subspace $\la \bu_1, \bu_2 \ra$ of $V(3, q^2)$ spanned by them contains exactly $q+1$ absolute points. The {\em Hermitian unital} $U_{H}(q)$ is defined to be the block design with point set $X$ in which a subset of $X$ is a block (called a {\em line}) precisely when it is the set of absolute points contained in some $\la \bu_1, \bu_2 \ra$. It is well known (see \cite{Kantor85, O'Nan, Taylor}) that $U_{H}(q)$ is a linear space with $q^3+1$ points, $q^2(q^2 - q + 1)$ lines, $q+1$ points in each line, and $q^2$ lines on each point. (Any linear space with these parameters is called a {\em unital}.) It was proved in \cite{O'Nan, Taylor} that $\Aut(U_{H}(q)) = \PGammaU(3, q)$. Thus, for every $G$ with $\PSU(3,q) \le G \le \PGammaU(3,q)$, $U_{H}(q)$ is a $G$-doubly point-transitive linear space. (Note that either $G = \PSU(3,q)$ or $G = \PGU(3, q) \rtimes \la \psi^{r} \ra$ for some divisor $r \ge 1$ of $2e$.) This implies that $G$ is also block-transitive and flag-transitive on $U_{H}(q)$, where a {\em flag} is an incident point-line pair.   

A line of $\PG(2, q^2)$ contains either one absolute point or $q+1$ absolute points. In the latter case the set of such $q+1$ absolute points is a line of $U_{H}(q)$; all lines of $U_{H}(q)$ are of this form. So we may represent a line of $U_{H}(q)$ by the homogenous equation of the corresponding line of $\PG(2, q^2)$. 

\begin{lemma}
\label{lem:soln}
Let $\bu_1 = (a_1, b_1, c_1)$ and $\bu_2 = (a_2, b_2, c_2)$ be isotropic vectors of $V(3, q^2)$ such that $\la \bu_1 \ra \ne \la \bu_2 \ra$. Then for each $\mu \in \FFF_{q^2}^*$ there are exactly $q+1$ vectors $\bu_0 \in V(3, q^2)$ such that  
\begin{equation}
\label{eq:rho}
\b(\bu_0, \bu_0) = -\mu^{q+1},\; \b(\bu_0, \bu_1) = 0,\; \b(\bu_0, \bu_2) = 0.
\end{equation} 
Moreover, any two such vectors $\bu_0$ differ by a scalar multiple. 
\end{lemma}

\begin{proof}
Since $\la \bu_1 \ra \ne \la \bu_2 \ra$, either $a_1 b_2 - a_2 b_1 \ne 0$, $a_1 c_2 - a_2 c_1 \ne 0$, or $b_1 c_2 - b_2 c_1 \ne 0$. 

Consider the case $b_1 c_2 - b_2 c_1 \ne 0$ first. Denote $\bu_0 = (a_0, b_0, c_0) \in V(3, q^2)$. From $\b(\bu_0, \bu_1) = \b(\bu_0, \bu_2) = 0$ we have
\begin{equation}
\label{eq:bc}
b_0 = a_0(a_2 b_1 - a_1 b_2)^q/(c_2 b_1 - c_1 b_2)^q,\;\, c_0 = a_0(a_2 c_1 - a_1 c_2)^q/(b_2 c_1 - b_1 c_2)^q. 
\end{equation}
Plug these into $\b(\bu_0, \bu_0) = -\mu^{q+1}$ we obtain $a_0^{q+1} = \mu^{q+1}/\eta$ for a certain $\eta \in \FFF^*_{q}$ determined by $\bu_1$ and $\bu_2$. Since $(\mu^{q+1}/\eta)^q = \mu^{q+1}/\eta$, we have $\mu^{q+1}/\eta \in \FFF^*_{q}$. Hence there are exactly $q+1$ elements $a_0 \in \FFF^*_{q^2}$ such that $a_0^{q+1} = \mu^{q+1}/\eta$ and so there are exactly $q+1$ vectors $\bu_0$ satisfying (\ref{eq:rho}). Because of (\ref{eq:bc}) any two such vectors are multiples of each other. 

The case where $a_1 b_2 - a_2 b_1 \ne 0$ or $a_1 c_2 - a_2 c_1 \ne 0$ can be dealt with similarly. 
\qed
\end{proof} 

Define
\begin{equation}
\label{eq:vq}
V(q) = \mbox{the set of flags of}\ \,U_{H}(q).
\end{equation}

\begin{definition}
\label{def:ugraph}
{\em Let $q = p^e > 2$ be a prime power and $r \ge 1$ a divisor of $2e$. Suppose $\l \in \FFF_{q^2}^{*}$ such that $\l^q$ belongs to the $\la \psi^r \ra$-orbit on $\FFF_{q^2}$ containing $\l$. 
The {\em unitary graph} $\Ga_{r,\l}(q)$ is defined to be the graph with vertex set $V(q)$ such that $(\la a_1, b_1, c_1 \ra, L_1)$, $(\la a_2, b_2, c_2 \ra, L_2) \in V(q)$ are adjacent if and only if $L_1$ and $L_2$ are given by:
\begin{equation}
\label{eq:l1}
L_1: \dmat{x & a_1 & a_0 + a_2\\y & b_1 & b_0 + b_2\\z & c_1 & c_0 + c_2} = 0
\end{equation}
\begin{equation}
\label{eq:l2}
L_2: \dmat{x & a_2 & a_0+\l^{qp^{ir}} a_1 \\y & b_2 & b_0 + \l^{qp^{ir}}b_1 \\z & c_2 & c_0 + \l^{qp^{ir}}c_1} = 0
\end{equation}
for an integer $0 \le i < 2e/r$ and a nonisotropic $(a_0, b_0, c_0) \in V(3, q^2)$ orthogonal to both $(a_1, b_1, c_1)$ and $(a_2, b_2, c_2)$.}
\end{definition}

\begin{remark}
{\em (a) The requirement on $(a_0, b_0, c_0)$ is equivalent to that $\bu_0 = (a_0, b_0, c_0)$ satisfies (\ref{eq:rho}) for $\bu_1 = (a_1, b_1, c_1)$ and $\bu_2 = (a_2, b_2, c_2)$, because every element of $\FFF_{q^2}^*$ is of the form $-\mu^{q+1}$ for some $\mu \in \FFF_{q^2}^*$. We will often use this fact in the sequel. 

(b) The requirement on $\l$ is equivalent to that $\l^{p^{tr}} = \l^q$ for at least one $0 \le t < 2e/r$. This ensures that $\Ga_{r,\l}(q)$ is defined as an undirected graph. (However, $\Ga_{r,\l}(q)$ is independent of the choice of such $t$.) In fact, since $r$ is a divisor of $2e$, we have $(j+t)r = 2e$ for some integer $j$ and so $\l = \l^{qp^{jr}}$. Hence the equations of $L_1$ and $L_2$ above can be rewritten as
\begin{equation}
\label{eq:l21}
L_2: \dmat{x & a_2 & \l a_0+\l^{qp^{ir}+1} a_1 \\y & b_2 & \l b_0 + \l^{qp^{ir}+1}b_1 \\z & c_2 & \l c_0 + \l^{qp^{ir}+1}c_1} = 0,\;\;\;
L_1: \dmat{x & \l^{qp^{ir}+1} a_1 & \l a_0+\l^{qp^{jr}} a_2 \\y & \l^{qp^{ir}+1} b_1 & \l b_0 + \l^{qp^{jr}} b_2 \\z & \l^{qp^{ir}+1} c_1 & \l c_0 + \l^{qp^{jr}} c_2} = 0.
\end{equation}
Since $\l \bu_0$ is a solution to (\ref{eq:rho}) with $\mu$ replaced by $\l \mu$, from (\ref{eq:l21}) it follows that the adjacency relation of $\Ga_{r,\l}(q)$ is symmetric and so $\Ga_{r,\l}(q)$ is well-defined as an undirected graph.

(c) $\Ga_{r,\l}(q)$ has $|V(q)| = q^2(q^3 + 1)$ vertices. Its valency is determined by $r$ and $\l$ (for a fixed $q$) as we will see in Theorem \ref{thm:char}.}
\end{remark}

\begin{example}
\label{q3} {\em In the case $q=3$, $r$ can be 1 or 2. If $r=1$, then every $\l \in \FFF^*_{9}$ trivially satisfies $\l^{3^{t}} = \l^3$ for $t=1$ and hence gives rise to the unitary graph $\Ga_{1,\l}(3)$. If $r=2$, then $t=0$, and $\l = \l^3$ holds if and only if $\l = 1$ or $\l = \omega^4 = 2$, where $\omega$ is a primitive element of $\FFF_9$. So $\Ga_{2,1}(3)$ and $\Ga_{2, 2}(3)$ are the only unitary graphs obtained from $r=2$. We will see that each $\Ga_{1,\l}(3)$ is $\PGammaU(3,3)$-symmetric, and $\Ga_{2,1}(3)$ and $\Ga_{2, 2}(3)$ are $\PGU(3,3)$-symmetric. With the help of MAGMA \cite{BCP} we obtain that $\Ga_{1,1}(3) \cong \Ga_{2,1}(3)$, $\Ga_{1,\omega}(3) \cong \Ga_{1,\omega^3}(3)$, $\Ga_{1,\omega^2}(3) \cong \Ga_{1,\omega^6}(3)$, $\Ga_{1, 2}(3) \cong \Ga_{2, 2}(3)$ and $\Ga_{1,\omega^5}(3) \cong \Ga_{1,\omega^7}(3)$, all with order $252$, and they have (valency, diameter, girth) = $(24, 3, 3), (48, 3, 3), (48, 2, 3), (24, 3, 3), (48, 3, 3)$, respectively.

Similarly, if $q=4$, the only graph-group pairs are: $(\Ga_{1,\l}(4), \PGammaU(3,4))$, $(\Ga_{2,\l}(4), \PGU(3,4) \rtimes \la \psi^2 \ra$), $(\Ga_{4,1}(4), \PGU(3,4))$ and $(\Ga_{4,\omega^5}(4), \PGU(3,4))$, where $\l \in \FFF^*_{16}$ and $\omega$ is a primitive element of $\FFF_{16}$.}
\end{example}

\begin{example}
\label{all} {\em For every divisor $r \ge 1$ of $e$, say, $e = tr$, $\l^{p^{tr}} = \l^q$ is trivially satisfied by all $\l \in \FFF^*_{q^2}$ and so $\Ga_{r,\l}(q)$ is well-defined. These graphs are $\PGU(3, q) \rtimes \la \psi^r \ra$-symmetric as we will see later.}
\end{example}

\section{Characterization of the unitary graphs}
\label{sec:char}

As part of the proof of Theorem \ref{thm:class}, in this section we characterize the unitary graphs as a certain family of imprimitive symmetric graphs admitting the unitary groups as groups of automorphisms. This characterization involves the quotient of a unitary graph with respect to the natural partition of its vertex set induced by the points of $U_{H}(q)$. A major tool to be used is the flag graph construction introduced in \cite{Zhou-EJC} which we outline below. 

Let $\DD$ be a $1$-design which admits a point- and block-transitive group 
$G$ of automorphisms. For two points $\s, \t$ of $\DD$ and a line $L$ incident with $\s$, denote by $G_{\s}$ the stabilizer of $\s$ in $G$, by $G_{\s \t}$ the stabilizer of $\s, \t$ in $G$ (subgroup of $G$ fixing each of $\s$ and $\t$), and by $G_{\s, L}$ the stabilizer of the flag $(\s, L)$. For a subset $\Om$ of flags of $\DD$, denote by $\Om(\s)$ the set of flags in $\Om$ with point entry $\s$. A $G$-orbit $\Om$ on the flags of $\DD$ is called {\em feasible with respect to $G$} \cite{Zhou-EJC} if 
\begin{description}
\item[{\rm (A1)}] $|\Om(\s)| \geq 3$; 

\item[{\rm (A2)}] $L \cap N = \{\s\}$, for distinct $(\s, L), (\s, N) \in \Om(\s)$;

\item[{\rm (A3)}] $G_{\s, L}$ is transitive on $L \setminus \{\s\}$, for $(\s, L) \in \Om$; 

\item[{\rm (A4)}] $G_{\s\t}$ is transitive on $\Om(\s) \setminus \{(\s, L)\}$, for 
$(\s, L) \in \Om$ and $\t \in L \setminus \{\s\}$.
\end{description}
Since $G$ is transitive on the points of $\DD$, the validity of these conditions is independent of the choice of $\s$.
Given a feasible $\Om$, a pair of flags $((\s, L), (\t, N)) \in 
\Om \times \Om$ is said to be {\em compatible} \cite{Zhou-EJC} with $\Om$ if 
\begin{description}
\item[{\rm (A5)}] $\s \not \in N$, $\t \not \in L$ but $\s \in N'$, $\t \in L'$
for some $(\s, L'), (\t, N') \in \Om$. 
\end{description}
From (A2) both $(\s, L')$ and $(\t, N')$ are uniquely determined by $((\s, L), (\t, N))$. If $\Psi \subseteq \Om \times \Om$ is a self-paired $G$-orbital of $\Om$ compatible with $\Om$, define \cite{Zhou-EJC} the {\em $G$-flag graph} of $\DD$ with respect to $(\Om, \Psi)$, denoted by $\Ga(\DD, \Om, \Psi)$, to be the graph with vertex set $\Om$ and arc set $\Psi$. It is proved in \cite[Theorem 1.1]{Zhou-EJC} that, for an imprimitive $G$-symmetric graph $(\Ga, \BB)$ with $\BB$ having block size at least $3$, $\Ga$ is an almost multicover of $\Ga_{\BB}$ if and only if $\Ga$ is isomorphic to $\Ga(\DD, \Om, \Psi)$ for a $G$-point-transitive and $G$-block-transitive $1$-design $\DD$. And in this case the block size of $\DD$ is equal to $m+1$ and $\Ga_{\BB}$ has valency $mv$ \cite[Lemma 2.1(a)]{Zhou-EJC}, where $m$ is the multiplicity of $\BB$. In particular, we have:

\begin{lemma}
\label{lem:flag graph}
(\cite[Corollary 2.6]{Zhou-EJC}) Let $s \geq 3$ be an integer and $G$ a finite group. The following statements are equivalent. 
\begin{itemize}
\item[\rm (a)] $\Ga$ is a $G$-symmetric graph admitting a nontrivial $G$-invariant 
partition $\BB$ of block size $s$ such that $\Ga_{\BB}$ is a complete graph and $\Ga$ is an almost multicover of $\Ga_{\BB}$. 
 
\item[\rm (b)] $\Ga$ is isomorphic to $\Ga(\DD,\Om,\Psi)$ for a $G$-doubly point-transitive and $G$-block-transitive $2$-$(v, k, \l)$ design $\DD$ with $(v - 1)/(k - 1) = s$, a feasible $G$-orbit 
$\Om$ on the flags of $\DD$, and a self-paired $G$-orbital $\Psi$ of $\Om$ compatible with $\Om$. 
\end{itemize}
Moreover, $v$ is equal to the number of vertices of $\Ga_{\BB}$, $k - 1$ is equal to the multiplicity of $\BB$, and $G$ is faithful on the vertex set of $\Ga$ if and only if it is faithful on the point set of $\DD$. 
\end{lemma}

The design $\DD$ in (b) corresponding to a given $(\Ga, G, \BB)$ is isomorphic to $\DD(\Ga, \BB)$ (see the introduction for its definition) as shown in the proof of \cite[Theorem 1.1]{Zhou-EJC}.  

We will need Lemma \ref{lem:flag graph} in the proof of Theorem \ref{thm:class}. At present we use it to prove the following characterization of the unitary graphs, which forms part of the proof of Theorem \ref{thm:class}. Denote by $B(\s)$ the set of flags of $U_{H}(q)$ with point-entry $\s$.  Define 
\begin{equation}
\label{eq:B}
\BB = \{B(\s): \s \in X\}, 
\end{equation}
so that $\BB$ is a partition of $V(q)$ with block size $q^2$. Denote by $L(\s \t)$ the unique line of $U_{H}(q)$ through two given points $\s$ and $\t$. For $r$ and $\l$ as in Definition \ref{def:ugraph}, define 
\begin{equation}
\label{eq:k}
k_{r,\l}(q) = \frac{|\la \psi^r \ra|}{|\la \psi^r \ra_{\l}|},
\end{equation}
where $\la \psi^r \ra_{\l}$ is the stabilizer of $\l$ in $\la \psi^r \ra$. Then $k_{r,\l}(q)$ is the size of the $\la \psi^r \ra$-orbit on $\FFF_{q^2}$ containing $\l$. Note that $k_{r,\l}(q)$ is a divisor of $2e/r$ and is equal to the least integer $j \ge 1$ such that $\l^{p^{jr}} = \l$. 

\begin{theorem}
\label{thm:char}
Let $q=p^e$ be a prime power and $r \ge 1$ a divisor of $2e$. Let $G = \PGU(3,q) \rtimes \la \psi^r \ra$. 
\begin{itemize}
\item[\rm (a)] Suppose $\l \in \FFF_{q^2}^{*}$ such that $\l^q$ belongs to the $\la \psi^r \ra$-orbit on $\FFF_{q^2}$ containing $\l$. Then $\Ga_{r,\l}(q)$ is a $G$-symmetric graph of order $q^2(q^3+1)$ and valency $k_{r,\l}(q)q(q^2-1)$ that admits $\BB$ above as a nontrivial $G$-invariant partition such that the quotient $\Ga_{r,\l}(q)_{\BB}$ is a complete graph and $\Ga_{r,\l}(q)$ is an almost multicover of $\Ga_{r,\l}(q)_{\BB}$. Moreover, for distinct points $\s, \t$ of $U_{H}(q)$, $(\s, L(\s \t))$ is the only vertex in $B(\s)$ which has no neighbour in $B(\t)$. Furthermore, the bipartite subgraph of $\Ga_{r,\l}(q)$ induced on $B(\s) \cup B(\t)$ (excluding the two isolates) has valency $k_{r,\l}(q)$, and each vertex of $\Ga_{r,\l}(q)$ has neighbours in exactly $q(q^2-1)$ blocks of $\BB$.
\item[\rm (b)] Conversely, if $\Ga$ is a $G$-symmetric graph that admits a nontrivial $G$-invariant partition $\BB$ of block size at least 3 such that $\Ga_{\BB}$ is a complete graph, $\Ga$ is an almost multicover of $\Ga_{\BB}$, and $\DD(\Ga, \BB)$ is a linear space, then $\Ga$ is isomorphic to a unitary graph $\Ga_{r,\l}(q)$ for some $q, r, \l$ as above.  
\end{itemize}
\end{theorem}
  
The rest of this section is devoted to the proof of Theorem \ref{thm:char}. We need the following results which can be easily verified. (We write elements of $\PGU(3,q)$ as linear transformations.)

\begin{lemma}
\label{lem:stab}
\begin{itemize}
\item[{\rm (a)}] $\PGU(3, q)_{\infty} = \{x' = (ax+bz)/a, y'=(ab^q x + a^{q+1}y + cz)/a, z' = z/a: a \in \FFF^*_{q^2}, b, c \in \FFF_{q^2}, c^q + c = b^{q+1}\}$.
\item[\rm (b)] $\PGU(3, q)_{\infty, 0} = \{x' = x, y'=a^{q}y, z' = z/a: a \in \FFF^*_{q^2}\}.$
\item[\rm (c)] For any divisor $r \ge 1$ of $2e$, $(\PGU(3, q) \rtimes \la \psi^{r} \ra)_{\infty} = \PGU(3, q)_{\infty} \rtimes \la \psi^{r} \ra$. 
\end{itemize}
\end{lemma}

In the following we will use the following well known facts: for any $\eta \in \FFF^*_{q}$, the equation $x^{q+1} = \eta$ has exactly $q+1$ solutions in $\FFF^*_{q^2}$; for any $\eta \in \FFF_{q}$, the equation $x+x^{q} = \eta$ has exactly $q$ solutions in $\FFF_{q^2}$. 

Now we prove that $V(q)$ is feasible with respect to $\PGU(3, q)$. 

\begin{lemma}
\label{lem:fbty}
\begin{itemize}
\item[{\rm (a)}] For any divisor $r \ge 1$ of $2e$, $V(q)$ is feasible with respect to $\PGU(3, q) \rtimes \la \psi^{r} \ra$.
\item[\rm (b)] If $3$ divides $q+1$, then $V(q)$ is not feasible with respect to $\PSU(3,q)$.
\end{itemize}
\end{lemma}
\begin{proof}
(a) It suffices to prove that $V(q)$ is feasible with respect to $\PGU(3, q)$. Obviously $V(q)$ satisfies (A1) and (A2). Let $L: x= 0$ be the line through $\infty$ and $0$. Since $\PGU(3, q)$ is transitive on $V(q)$, to prove (A3) it suffices to prove that $\PGU(3,q)_{\infty, L}$ is transitive on $L \setminus \{\infty\}$. In fact, by Lemma \ref{lem:stab} we have 
\begin{equation}
\label{eq:infL}
\PGU(3, q)_{\infty, L} = \{x' = x, y'=(a^{q+1}y + cz)/a, z' = z/a: a \in \FFF^*_{q^2}, c^q + c = 0\}.
\end{equation} 
By the orbit-stabilizer lemma, the $\PGU(3, q)_{\infty, L}$-orbit on $L \setminus \{\infty\}$ containing $0$ has length $|\PGU(3, q)_{\infty, L}|/|\PGU(3, q)_{\infty, L, 0}| = q = |L \setminus \{\infty\}|$ (note that $\PGU(3, q)_{\infty, L, 0}$ $= \PGU(3, q)_{\infty, 0}$). It follows that $\PGU(3,q)_{\infty, L}$ is transitive on $L \setminus \{\infty\}$.

Since $\PGU(3, q)$ is doubly transitive on the set of points of $U_{H}(q)$, $V(q)$ satisfies (A4) with respect to $\PGU(3, q)$ if and only if $\PGU(3,q)_{\infty, 0}$ is transitive on the set of lines through $0$ other than $L$. Fix such a line, say, $N: y = x$. The image of $N$ under a typical element $x' = x, y'=a^{q}y, z' = z/a$ of $\PGU(3,q)_{\infty, 0}$ has equation $y = a^q x$, where $a \in \FFF^*_{q^2}$. Since all lines of $U_{H}(q)$ through $0$ other than $L$ are of this form, it follows that $V(q)$ satisfies (A4) and so is feasible with respect to $\PGU(3, q)$. 

(b) By Lemma \ref{lem:stab}, the image of $N: y = x$ under a typical element of $\PSU(3,q)_{\infty, 0}$ has equation $y = a^q x$, where $a \in \FFF^*_{q^2}$ satisfies $a^{q-1} = d^3$ for some $d \in \FFF^*_{q^2}$ such that $d^{q+1} = 1$. Since not every line of $U_{H}(q)$ through $0$ but not $\infty$ is of this form when $3$ divides $q+1$, $V(q)$ is not feasible with respect to $\PSU(3, q)$ in this case. 
\qed
\end{proof}

\begin{lemma}
\label{lem:3}
For $(\s, L) \in V(q)$, $\PGU(3,q)_{\s, L}$ is regular on the set of points of $U_{H}(q)$ not in $L$. 
\end{lemma}
\begin{proof}
Since $\PGU(3,q)$ is transitive on $V(q)$, without loss of generality we may assume $\s = \infty$ and $L: x= 0$, so that $\PGU(3, q)_{\infty, L}$ is as given in (\ref{eq:infL}). Fix a point $\t = \la l, m, 1 \ra$ of $U_{H}(q)$ not in $L$, where $l \in \FFF^*_{q^2}$ and $m \in \FFF_{q^2}$ such that $l^{q+1} = m + m^q$. An element $x' = x, y'=(a^{q+1}y + cz)/a, z' = z/a$ of $\PGU(3, q)_{\infty, L}$ fixes $\t$ if and only if $a = 1$ and $c = 0$; that is, $\PGU(3, q)_{\infty, L, \t}$ is the identity subgroup of $\PGU(3, q)$. Thus the $\PGU(3, q)_{\infty, L}$-orbit on $V(q)$ containing $\t$ has length $|\PGU(3, q)_{\infty, L}| = (q^2-1)q$. Since there are exactly $(q^2-1)q$ points of $U_{H}(q)$ not in $L$, it follows that $\PGU(3,q)_{\infty, L}$ is regular on such points. 
\qed 
\end{proof}

\medskip
\begin{proof} \textbf{of Theorem \ref{thm:char}}~
(a) Denote $\Ga = \Ga_{r,\l}(q)$ and $k^* = k_{r,\l}(q)$. Since $\PGU(3, q)$ preserves the Hermitian form $\b$, one can verify that $G$ preserves the adjacency relation of $\Ga$ (under the induced action of $G$ on $V(q)$) and hence can be viewed as a subgroup of $\Aut(\Ga)$. $\Ga$ is clearly $G$-vertex transitive. We now prove that $\Ga$ is $G$-symmetric.  
 
Choose $d \in \FFF^*_{q^2}$ such that $d + d^q = 1$. Then $\la 1, d, 1 \ra \in X$ and $L: x = z$ is the unique line through $\infty$ and $\la 1, d, 1 \ra$. Since an element of $G_{\infty}$ fixes $L$ if and only if it maps $\la 1, d, 1 \ra$ to a point in $L$, we have
\begin{equation}
\label{eq:H}
H := G_{\infty, L} = \PGU(3, q)_{\infty, L} \rtimes \la \psi^r \ra,
\end{equation}
where by Lemma \ref{lem:stab} one can show that $\PGU(3, q)_{\infty, L} = \{x' = (ax+(1-a)z)/a, y'=(a(1-a)^q x + a^{q+1}y + cz)/a, z' = z/a: a \in \FFF^*_{q^2}, c^q + c = (1-a)^{q+1}\}$.

A flag $(\la \bu_2 \ra, L_2) \in V(q)$ is adjacent to $(\infty, L)$ in $\Ga$ if and only if there exists $\bu_0 = (a_0, b_0, c_0) \in V(3, q^2)$ satisfying (\ref{eq:rho}) for $\bu_1 =  (0, 1, 0)$, $\bu_2 = (a_2, b_2, c_2)$ and some $\mu \in \FFF^*_{q^2}$ such that $L$ and $L_2$ are given by (\ref{eq:l1}) and (\ref{eq:l2}) respectively. One can see that (\ref{eq:l1}) gives $L: x=z$ if and only if $c_0 + c_2 = a_0 + a_2 \ne 0$. From the second equation in (\ref{eq:rho}) we have $c_0 = 0$, and using this the other two equations in (\ref{eq:rho}) amount to $a_0^{q+1} = \mu^{q+1}$ and $a_0 a_2^q = b_0 c_2^q$ respectively. Thus $\bu_0 = (a_0, a_0 a_2^q/(a_0 + a_2)^q, 0)$. Since $\bu_2$ is assumed to be isotropic, we have $(a_0 + a_2)^q b_2 + (a_0 + a_2) b_2^q = a_2^{q+1}$, which gives exactly $q$ possible values of $b_2$ for fixed $a_0$ and $a_2$. Note that for a fixed $a_0$, $a_2$ can be any element of $\FFF_{q^2}$ other than $-a_0$. Thus, for a fixed $a_0$, there are exactly $q(q^2-1)$ 
isotropic vectors $\bu_2 = (a_2, b_2, a_0 + a_2)$, of which no two are multiples of each other, such that $(\la \bu_2 \ra, L_2)$ is adjacent to $(\infty, L)$ for some $L_2$ through $\bu_2$. Since for a fixed $\mu$ any two solutions to $x^{q+1} = \mu^{q+1}$ differ by a scalar multiple which is a $(q+1)$st root of unity, one can see that different pairs $(a_0, \mu)$ satisfying $a_0^{q+1} = \mu^{q+1}$ give rise to the same set of $\bu_2$. Thus it suffices to consider one fixed pair $(a_0, \mu)$ only. Since $\l^{qp^{ir}} = \l^{qp^{i' r}}$ if and only if $\l^{p^{ir}} = \l^{p^{i'r}}$, each $\bu_2$ corresponds to precisely $k^*$ different lines $L_2$ such that $(\la \bu_2 \ra, L_2)$ is adjacent to $(\infty, L)$.
Therefore, the valency of $(\infty, L)$ (and hence all other vertices) in $\Ga$ is equal to $k^* q(q^2-1)$. 
 
Choosing $\mu = 1$, $\bu_0 = (1, 0, 0)$, $\bu_2 = (0, 0, 1)$ and $i = 0$,  (\ref{eq:l2}) gives rise to the line $N: y=\l^q x$ and thus $(\infty, L)$ and $(0, N)$ are adjacent in $\Ga$. We have $H_0 = \la \psi^r \ra$, where $H$ is as defined in (\ref{eq:H}). An element $\psi^{jr}$ of $H_0$ maps $N$ to the line $y = \l^{qp^{jr}}x$ and hence it fixes $N$ if and only if $\l^{p^{jr}} = \l$. In other words, $H_{0, N} = \{\psi^{jr}: 0 \le j < 2e/r, \l^{p^{jr}} = \l\}$, which is the stabilizer of $\l$ in $\la \psi^r \ra \le \Aut(\FFF_{q^2})$. Thus $|H_{0, N}| = 2e/(k^*r)$. Since $|H| = 2eq(q^2 - 1)/r$, by the orbit-stabilizer lemma the $H$-orbit on $V(q)$ containing $(0, N)$ has size $k^* q(q^2 - 1)$, which is equal to the valency of $(\infty, L)$ in $\Ga$ by the previous paragraph. Since $H$ is a subgroup of $\Aut(\Ga)$ fixing $(\infty, L)$ and $(0, N)$ is adjacent to $(\infty, L)$, the $H$-orbit on $V(q)$ containing $(0, N)$ is contained in the neighbourhood of $(\infty, L)$ in $\Ga$. Since the two sets have the same size, it follows that $H$ is transitive on the neighbourhood of $(\infty, L)$ in $\Ga$. This together with the $G$-vertex transitivity of $\Ga$ implies that $\Ga$ is $G$-symmetric.  

It is clear that $\BB$ is a $G$-invariant partition of the vertex set $V(q)$ of $\Ga$. 
Since $G_{\infty, L, 0} = H_0 = \la \psi^r \ra$ and $G_{\infty, 0} = \PGU(3,q)_{\infty, 0} \rtimes \la \psi^r \ra$, by Lemma \ref{lem:stab}(b) the $G_{\infty, 0}$-orbit containing $L$ on the set of lines through $\infty$ has size $q^2 - 1$. Since $\Ga$ is $G$-symmetric and $(\infty, L)$ is adjacent to $(0, N)$, it follows that exactly $q^2 - 1$ vertices in $B(\infty)$ have neighbours in $B(0)$. Since $G$ is doubly transitive on the points of $U_H(q)$, this implies that for every pair of distinct blocks $B, C \in \BB$, exactly $q^2 - 1$ vertices of $B$ have neighbours in $C$. Let $L_0: x = 0$ be the unique line of $U_{H}(q)$ through $\infty$ and $0$. One can verify that $(\infty, L_0)$ is not adjacent to any vertex in $B(0)$ and $(0, L_0)$ is not adjacent to any vertex in $B(\infty)$. Thus $(\infty, L_0)$ is the only vertex in $B(\infty)$ without neighbour in $B(0)$, and $(0, L_0)$ is the only vertex in $B(0)$ without neighbour in $B(\infty)$. Since $G$ is doubly transitive on $X$, the last statement in (a) follows. 

Since the $H_0$-orbit containing $N$ on the lines of $U_{H}(q)$ has size $|H_0|/|H_{0, N}| = k^*$, $(\infty, L)$ has exactly $k^*$ neighbours in $B(0)$. Since $\Ga$ is $G$-symmetric and $\BB$ is $G$-invariant, it follows that any block $B \in \BB$ contains either none or exactly $k^*$ neighbours of any vertex not in $B$. Since the valency of $\Ga$ is equal to $k^* q(q^2-1)$ as shown above, each vertex of $\Ga$ has neighbours in exactly $q(q^2-1)$ blocks of $\BB$.   

\delete{
Counting the number of pairs $((\infty, L_1), B(\s))$ such that $(\infty, L_1) \in B(\infty)$ has at least one neighbour in $B(\s)$, we obtain that each vertex of $B(\infty)$ (and hence of $\Ga$) has neighbours in exactly $q(q^2-1)$ blocks of $\BB$. Since the valency of $\Ga$ is equal to $k^*q(q^2-1)$ as shown above, it follows that any block of $\BB$ containing a neighbour of a vertex must contain exactly $k^*$ neighbours of that vertex.    
}

(b) By Lemma \ref{lem:flag graph}, we have $\Ga \cong \Ga(\DD,\Om,\Psi)$ for a $G$-doubly point-transitive linear space $\DD$, a feasible $G$-orbit $\Om$ on the flags of $\DD$, and a self-paired $G$-orbital $\Psi$ of $\Om$ compatible with $\Om$. From the classification \cite{Kantor85} of such linear spaces it follows that $\DD = U_H(q)$ and so $\Om = V(q)$ by Lemma \ref{lem:fbty}.  We will determine all possible $\Psi$ and prove that they produce precisely the unitary graphs.  

Let $L: x= z$ be as above. Since $0$ is not on $L$, by Lemma \ref{lem:3} any $G$-orbital of $V(q)$ should contain $((\infty, L), (0, N))$ for some line $N$ through $0$ but not $\infty$. Suppose this $G$-orbital is self-paired. Then there exists $\phi \in G$ that interchanges $(\infty, L)$ and $(0, N)$. Since $\phi_0: x'=x, y' = z, z' = y$ is in $\PGU(3, q)$ and interchanges $\infty$ and $0$, it follows that $\phi \phi_0 \in G_{\infty, 0}$. Thus, by Lemma \ref{lem:stab}(b), $\phi \phi_0$ is of the form: $x' =  x^{p^{kr}}, y'=\l^{q}y^{p^{kr}}, z' = z^{p^{kr}}/\l$ for some $\l \in \FFF^*_{q^2}$ and $0 \le k < 2e/r$. Hence $\phi$ is of the form
$$
\phi: x' = x^{p^{kr}}, y'=\l^{q}z^{p^{kr}}, z' = y^{p^{kr}}/\l.
$$
It follows that $\phi^2: x' = x^{p^{2kr}}, y'=\l^{q-p^{kr}}y^{p^{2kr}}, z' = \l^{qp^{kr}-1}z^{p^{kr}}$. Since $\phi^2$ fixes $L$ and maps $\la 1, d, 1 \ra \in L$ to $\la 1, \l^{q-p^{kr}}d^{p^{2kr}}, \l^{qp^{kr}-1} \ra$, where $d \in \FFF_{q^2}^*$ is such that $d + d^q = 1$, we have $\l^{qp^{kr}-1} = 1$, that is, $\l^{p^{kr}} = \l^{q}$. Conversely, one can see that if $\phi$ above satisfies $\l^{p^{kr}} = \l^{q}$ then it interchanges $(\infty, L)$ and $(0, N)$ and hence the $G$-orbital of $V(q)$ containing $((\infty, L), (0, N))$ is self-paired, where $N: y = \l^q x$ is the image of $L$ under $\phi$. 

The argument above shows that each self-paired $G$-orbital $\Psi$ of $V(q)$ should contain $((\infty, L), (0, N))$ for some $N: y = \l^q x$ which is the image of $L$ under some $\phi$ above such that $\l^{p^{kr}} = \l^{q}$ for some $k$, and vice versa. Since $0 \not \in L$ and $\infty \not \in N$, $\Psi$ is compatible with $V(q)$ (taking for example $L' = N'$ to be the unique line through $\infty$ and $0$ in (A5)) and hence gives rise to the $G$-flag graph $\Ga(U_H(q), V(q), \Psi)$. Moreover, all $G$-flag graphs of $U_H(q)$ are of this form. 

We now prove that $\Ga(U_H(q), V(q), \Psi)$ above is isomorphic to $\Ga_{r,\l}(q)$. Let $\bu_1 = (a_1, b_1, c_1)$ and $\bu_2 = (a_2, b_2, c_2)$ be isotropic vectors with $\la \bu_1 \ra \ne \la \bu_2 \ra$. If an element of $G$ maps $\infty, 0$ to $\la \bu_1 \ra, \la \bu_2 \ra$ respectively, then it must be of the form $\phi_{\bu_0, i}: x' = a_0x^{p^{ir}}+a_1 y^{p^{ir}}+a_2 z^{p^{ir}}, y' = b_0x^{p^{ir}}+b_1 y^{p^{ir}}+b_2 z^{p^{ir}}, z' = c_0x^{p^{ir}}+c_1 y^{p^{ir}}+c_2 z^{p^{ir}}$ for some $\bu_0 = (a_0, b_0, c_0) \in V(3, q^2)$ and $0 \le i < 2e/r$. Scaling when necessary, we may assume that $\phi_{\bu_0, i}$ maps $(0, 1, 0), (0, 0, 1)$ to $\bu_1, \bu_2$ respectively, so that $\b(\bu_1, \bu_2) = \b((0, 1, 0), (0, 0, 1)) = 1$. Denote by $A$ the matrix of $\phi_{\bu_0, i}$ with columns $\bu_0, \bu_1$ and $\bu_2$. Since $\phi_{\bu_0, i} \in G$, we have $A^T D \bar{A} = D$, or equivalently $\b(\bu_0, \bu_0) = -1, \b(\bu_0, \bu_1) = 0, \b(\bu_0, \bu_2) = 0$. By Lemma \ref{lem:soln} these conditions are satisfied by precisely $q+1$ vectors $\bu_0 = (a_0, b_0, c_0)$ that differ by scalar multiples. Thus the arcs of $\Ga(U_H(q), V(q), \Psi)$ are precisely those $((\la \bu_1\ra, L_1), (\la \bu_2 \ra, L_2))$ such that $\bu_1$ and $\bu_2$ are as above, $\bu_0$ is a solution to (\ref{eq:rho}) with $\mu = 1$, and $L_1, L_2$ are respectively the images of $L, N$ under $\phi_{\bu_0, i}$ for some $i$. On the other hand, since $L$ has equation $x=z$, one can verify that $L_1$ has equation (\ref{eq:l1}). Similarly, since $N$ has equation $y = \l^q x$, $L_2$ has equation (\ref{eq:l2}). Therefore, $\Ga(U_H(q), V(q), \Psi) \cong \Ga_{r,\l}(q)$.
\qed
\end{proof}

The neighbourhood in $\Ga_{r,\l}(q)$ of a subset of $V(q)$ is defined as the union of the neighbourhoods of its vertices in $\Ga_{r,\l}(q)$.
As a consequence of Theorem \ref{thm:char}, the neighbourhood of each $B(\s) \in \BB$ in $\Ga_{r,\l}(q)$ is equal to the set of flags $\{(\t, N) \in V(q): \s \not \in N\}$ of $U_{H}(q)$ whose lines do not pass through $\s$.

\section{Another quotient of the unitary graphs}  
\label{sec:2nd}

Denote by $C(L)$ the set of flags of $U_{H}(q)$ with line-entry $L$. Let
\begin{equation}
\label{eq:C}
\CC = \{C(L):\;\mbox{$L$ is a line of $U_{H}(q)$}\}.
\end{equation}
Obviously, $\CC$ is a partition of $V(q)$ with block size $q+1$. Unlike $\Ga_{r,\l}(q)_{\BB} \cong K_{q^3 + 1}$, the quotient graph $\Ga_{r,\l}(q)_{\CC}$ of $\Ga_{r,\l}(q)$ relative to $\CC$ is not a complete graph. This section is devoted to combinatorial properties of $\Ga_{r,\l}(q)_{\CC}$ in relation to $\Ga_{r,\l}(q)$. 

If $\l \ne 1$, denote
\begin{equation}
\label{eq:l}
\eta = \left(\frac{1 - \l}{\l}\right)^{q+1},\;\; l_{r,\l}(q) = \frac{|\la \psi^r \ra|}{|\la \psi^r \ra_{\eta}|}.
\end{equation} 
Then $l_{r,\l}(q)$ is the size of the $\la \psi^r \ra$-orbit on $\FFF_{q^2}$ containing $\eta$. Straightforward computation yields $\la \psi^r \ra_{\l} \le \la \psi^r \ra_{\eta}$ and hence $l_{r,\l}(q)$ is a divisor of $k_{r,\l}(q)$. The following is the main result in this section. (The {\em lexicographic product} of a graph $\Sigma$ with a graph $\Delta$ is defined to have vertex set $V(\Si) \times V(\Delta)$ such that $(v, w), (v', w')$ are adjacent if and only if either $v, v'$ are adjacent in $\Si$ or $v = v'$ and $w, w'$ are adjacent in $\Delta$.)

\begin{theorem}
\label{thm:2ndquotient}
Let $q=p^e$ be a prime power and $r \ge 1$ a divisor of $2e$. Let $G = \PGU(3,q) \rtimes \la \psi^r \ra$. Suppose $\l \in \FFF_{q^2}^{*}$ such that $\l^q$ belongs to the $\la \psi^r \ra$-orbit on $\FFF_{q^2}$ containing $\l$. Denote $\Ga = \Ga_{r,\l}(q)$, $k^* = k_{r,\l}(q)$ and $l^* = l_{r,\l}(q)$. Then the following hold:
\begin{itemize}
\item[\rm (a)] the valency of $\Ga_{\CC}$ is equal to $q(q-1)$ if $\l=1$, $q(q^2-1)l^*$ if $\l \ne 1$ and $\l + \l^q \ne 1$, and $(q+1)(q^2-1)$ if $\l + \l^q = 1$;
\item[\rm (b)] the number of blocks of $\CC$ containing at least one neighbour of a fixed vertex of $\Ga$ is equal to $q(q-1)$ if $\l=1$, $q(q^2-1)l^*$ if $\l \ne 1$ and $\l + \l^q \ne 1$, and $q(q^2-1)$ if $\l + \l^q = 1$;
\item[\rm (c)] if $\l=1$, then $\Ga$ is isomorphic to the lexicographic product of $\Ga_{\CC}$ and the empty graph of $q+1$ vertices; if $\l \ne 1$ and $\l + \l^q \ne 1$, then $\Ga$ is a multicover of $\Ga_{\CC}$ and for adjacent $C(L_1), C(L_2) \in \CC$ the valency of the bipartite subgraph of $\Ga$ induced on $C(L_1) \cup C(L_2)$ is equal to $k^*/l^*$; if $\l + \l^q = 1$, then $\Ga$ is an almost multicover of $\Ga_{\CC}$ and the valency of this bipartite graph (excluding the two isolates) is equal to $k^*$. 
\end{itemize}
\end{theorem}

We need the following two lemmas in the proof of Theorem \ref{thm:2ndquotient}.  

\begin{lemma}
\label{lem:1edge}
Under the assumption of Theorem \ref{thm:2ndquotient}, let $d \in \FFF_{q^2}$ be such that $d + d^q = 1$ and $d \ne 1-\l$. Let $L: x = 0$ and $N: (1-\l)x + y +  (1-\l-d)z = 0$ be lines of $U_{H}(q)$. Then $(\infty, L)$ and $(\tau, N)$ are adjacent in $\Ga_{r,\l}(q)$, where $\tau = \la -1, d, 1 \ra$. 
\end{lemma}

\begin{proof}
Choose $\bu_0 = (1, -1, 0), \bu_1 = (0, 1, 0), \bu_2 = (-1, d, 1), \mu = 1$ and $i = (2e/r) - k$ in Definition \ref{def:ugraph} and (\ref{eq:rho}). 
\qed
\end{proof}

\begin{lemma}
\label{lem:2linestab}
Under the condition of Lemma \ref{lem:1edge}, the following hold for $G = \PGU(3,q) \rtimes \la \psi^r \ra$:
\begin{itemize}
\item[\rm (a)] $|G_{L}| = \frac{2q(q+1)(q^2-1)e}{r}$
\item[\rm (b)] $|G_{L, N}| = \left\{ 
\begin{array}{lr}
\frac{2(q+1)^2 e}{r}, & \mbox{if $\l = 1$}\\ [0.3cm]
\frac{2(q+1)e}{l^* r}, & \mbox{if $\l \ne 1$ and $\l + \l^q \ne 1$}\\ [0.3cm]
\frac{2qe}{r}, & \mbox{if $\l + \l^q = 1$}
\end{array}
\right.$
\item[\rm (c)] $|G_{\infty, L, N}| = \left\{ \begin{array}{lr}
\frac{2(q+1)e}{r}, & \mbox{if $\l = 1$}\\ [0.3cm]
\frac{2e}{l^* r}, & \mbox{if $\l \ne 1$ and $\l + \l^q \ne 1$}\\ [0.3cm]
\frac{2e}{r}, & \mbox{if $\l + \l^q = 1$}
\end{array}
\right.$
\end{itemize}
\end{lemma}

\begin{proof}
Since $G$ is transitive on the lines of $U_{H}(q)$, we have $|G_{L}| = |G|/q^2(q^2-q+1) = 2q(q+1)(q^2-1)e/r$ and this proves (a). 

To prove (b) and (c) we need more information about $G_L$. Consider an element $\zeta_{A, j}: x' = a_1 x^{p^{jr}} + a_2 y^{p^{jr}} + a_3 z^{p^{jr}}, y' = b_1 x^{p^{jr}} + b_2 y^{p^{jr}} + b_3 z^{p^{jr}}, z' = c_1 x^{p^{jr}} + c_2 y^{p^{jr}} + c_3 z^{p^{jr}}$ of $G$, where $A = \bmat{a_1 & a_2 & a_3\\b_1 & b_2 & b_3\\c_1 & c_2 & c_3}$ is the corresponding matrix. Since $\zeta_{A, j}$ maps $\infty, 0 \in L$ to $\la a_2, b_2, c_2 \ra, \la a_3, b_3, c_3 \ra$ respectively, $\zeta_{A, j}$ fixes $L$ if and only if $a_2 = a_3 = 0$. This together with $A^T D \bar{A} = D$ implies that $\zeta_{A, j} \in G_{L}$ if and only if
\bea
-a_1^{q+1} + b_1 c_1^q + c_1 b_1^q = -1 \label{eq:11} \\
b_1 c_2^q + c_1 b_2^q = 0 \label{eq:12} \\
b_1 c_3^q + c_1 b_3^q = 0 \label{eq:13} \\
b_2 c_2^q + c_2 b_2^q = 0 \label{eq:22} \\
b_2 c_3^q + c_2 b_3^q = 1 \label{eq:23} \\
b_3 c_3^q + c_3 b_3^q = 0 \label{eq:33}
\eea 
If none of $b_1, b_2$ and $b_3$ is equal to $0$, then $c_1/b_1 = -(c_2/b_2)^q = c_2/b_2$ and $c_1/b_1 = -(c_3/b_3)^q = c_3/b_3$. Call this common ratio $h$. Then $h+h^q = 0$ but $(h+h^q)b_2b_3^q = 1$ by (\ref{eq:23}). This contradiction shows that at least one of $b_1, b_2$ and $b_3$ is equal to $0$. On the other hand, by (\ref{eq:23}) at least one of $b_2, b_3$ is non-zero. If $b_1 \ne 0$, $b_2 = 0$ but $b_3 \ne 0$, then $c_2 = 0$ by (\ref{eq:12}), which contradicts (\ref{eq:23}). Similarly, if $b_1 \ne 0$,  $b_3 = 0$ but $b_2 \ne 0$, then $c_3 = 0$ by (\ref{eq:13}), which again contradicts (\ref{eq:23}). Therefore, we must have $b_1 = 0$, which implies $a_1^{q+1} = 1$ by (\ref{eq:11}) and $c_1 = 0$ by (\ref{eq:12}) and (\ref{eq:13}). Scaling when necessary, we may assume $a_1 = 1$ in the following.

If $b_2, b_3 \ne 0$, then by (\ref{eq:22}) and (\ref{eq:33}), we have $c_2 = ab_2$ and $c_3 = b b_3$ for some $a, b \in \FFF_{q^2}$ with $a+a^q=b+b^q=0$. From this and denoting $b_3$ by $c$, we get $b_2 = (a-b)^{-1}c^{-q}$ by (\ref{eq:23}). Hence $A$ is of the form 
\begin{equation}
\label{eq:form1}
\bmat{1 & 0 & 0\\0 & (a-b)^{-1}c^{-q} & c\\ 0 & a(a-b)^{-1}c^{-q} & bc}, 
\end{equation}
where $a, b, c \in \FFF_{q^2}$ such that $a+a^q=b+b^q=0, a \ne b$ and $c \ne 0$.
 
If $b_2 = 0$ but $b_3 \ne 0$, then $c_2 = b_3^{-q}$ by (\ref{eq:23}) and $(c_3/b_3)+(c_3/b_3)^q = 0$ by (\ref{eq:33}). In this case $A$ is of the form 
\begin{equation}
\label{eq:form2}
\bmat{1 & 0 & 0\\0 & 0 & b\\ 0 & b^{-q} & ab}, 
\end{equation}
where $a \in \FFF_{q^2}$ and $b \in \FFF^*_{q^2}$ such that $a+a^q=0$. 
Similarly, if $b_2 \ne 0$ but $b_3 = 0$, then $A$ is of the form 
\begin{equation}
\label{eq:form3}
\bmat{1 & 0 & 0\\0 & b & 0\\ 0 & ab & b^{-q}},
\end{equation}
where $a \in \FFF_{q^2}$ and $b \in \FFF^*_{q^2}$ such that $a+a^q=0$. 
Note that (\ref{eq:form1})-(\ref{eq:form3}) give rise to all elements $\zeta_{A, j}$ of $G_L$, where $0 \le j < 2e/r$. 

\delete{
$G_{L} = \{x' = x^{p^{jr}}, y' = (a+b^q)^{-1}c^{-q}y^{p^{jr}}+cz^{p^{jr}}, z' = a(a+b^q)^{-1}c^{-q}y^{p^{jr}}+bcz^{p^{jr}}: a, b \in \FFF_{q^2}, c \in \FFF_{q^2}^{*}, a+a^q=b+b^q=0, a+b^q \ne 0, 0 \le j < 2e/r\} \cup \{x' = x^{p^{jr}}, y' = bz^{p^{jr}}, z' = b^{-q}y^{p^{jr}}+abz^{p^{jr}}: a \in \FFF_{q^2}, b \in \FFF_{q^2}^{*}, a+a^q=0, 0 \le j < 2e/r\} \cup \{x' = x^{p^{jr}}, y' = by^{p^{jr}}, z' = aby^{p^{jr}}+b^{-q}z^{p^{jr}}: a \in \FFF_{q^2}, b \in \FFF_{q^2}^{*}, a+a^q=0, 0 \le j < 2e/r\}$.
}

\smallskip
{\bf Case 1:}~$\l = 1$. 
\smallskip

In this case $N$ has equation $y = dz$.

\smallskip
{\sc Claim 1.1:}~$G_{L, N}$ consists of those $\zeta_{A, j}$ such that $0 \le j < 2e/r$ and $A$ is in one of the following forms:
\begin{itemize}
\item[(i)] $A$ is given by (\ref{eq:form1}) with $a = -(bd-1)^{q+1} [(bd-1)^q d + d^{qp^{jr}}]^{-1} + b$ and $c^{q+1} = d^{(q+1)p^{jr}}(bd-1)^{-(q+1)}$, where $b \in \FFF_{q^2}$ is such that $b+b^q=0$ and $b \ne d^{-1} - d^{-(q+1)+p^{jr}}$;
\item[(ii)] $A$ is given by (\ref{eq:form2}) with $a = d^{-1} - d^{-(q+1)+p^{jr}}$ and $b^{q+1} = d^{q+1}$;
\item[(iii)] $A$ is given by (\ref{eq:form3}) with $a = d^{-1} - d^{-(q+1)+qp^{jr}}$ and $b^{q+1} = d^{-(q+1)(p^{jr}-1)}$.
\end{itemize}
\smallskip 

{\sc Claim 1.2:}~$G_{\infty, L, N}$ consists of those $\zeta_{A, j}$ such that
\begin{itemize}
\item[(iv)] $0 \le j < 2e/r$ with $d^{p^{jr}} \ne d$, and $A$ is given by (\ref{eq:form1}) with $a=0$, $b=(d-d^{p^{jr}})^{-1}$ and $c^{q+1} = (d - d^{p^{jr}})^{q+1}$; or
\item[(v)] $0 \le j < 2e/r$ with $d^{p^{jr}} = d$, and $A$ is given by (\ref{eq:form3}) with $a=0$ and $b^{q+1}=1$.
\end{itemize}
\smallskip

Since the numbers of elements of $G_{L, N}$ of types (i)-(iii) are $(q-1)(q+1)(2e/r), (q+1)(2e/r), (q+1)(2e/r)$, respectively, we obtain $|G_{L, N}| = 2(q+1)^2 e/r$ from Claim 1.1 and $|G_{\infty, L, N}| = 2(q+1)e/r$ from Claim 1.2.

\smallskip
\textit{Proof of Claims 1.1 and 1.2:}~One can verify that the image of $N$ under $\zeta_{A, j} \in G_L$ with $A$ given by (\ref{eq:form1}) has equation $[(a+b^q)bc^{q+1}+ad^{p^{jr}}]y = [(a+b^q)c^{q+1}+d^{p^{jr}}]z$. Thus $N$ is fixed by $\zeta_{A, j}$ if and only if $(a+b^q)c^{q+1}+d^{p^{jr}} = d[(a+b^q)bc^{q+1}+ad^{p^{jr}}]$, or equivalently $a = [d^{p^{jr}} - b^q c^{q+1}(bd-1)] [d^{p^{jr}+1}+c^{q+1}(bd-1)]^{-1}$. Plugging this into $a+a^q = 0$, and using $b+b^q=0$ and $d+d^q=1$, we obtain 
$c^{q+1} = d^{(q+1)p^{jr}}(bd-1)^{-(q+1)}$. (Note that $bd \ne 1$.) Hence $a = [(bd-1)^q + b d^{qp^{jr}}] [(bd-1)^q d + d^{qp^{jr}}]^{-1} = -(bd-1)^{q+1} [(bd-1)^q d + d^{qp^{jr}}]^{-1} + b$. One can see that $(bd-1)^q d + d^{qp^{jr}} = 0$ if and only if $b = d^{-1} - d^{-(q+1)+p^{jr}}$, and on the other hand this particular $b$ satisfies $b + b^q = 0$. Therefore, an element $\zeta_{A, j}$ of $G_L$ with $A$ given by (\ref{eq:form1}) fixes $N$ if and only if $a, b$ and $c$ are as in (i). Such an element of $G_{L, N}$ maps $\infty$ to $(0, (a+b^q)^{-1}c^{-q}, a(a+b^q)^{-1}c^{-q})$. Hence it fixes $\infty$ if and only if $a = 0$, or equivalently $b=(d-d^{p^{jr}})^{-1}$. Here we require $d^{p^{jr}} \ne d$. Note that this $b$ satisfies $b+b^q=0$ and is distinct from $d^{-1} - d^{-(q+1)+p^{jr}}$ because $d^{p^{jr}} \ne d - d^q$. (In fact, if $d^{p^{jr}} = d - d^q$, then $(d+ d^{q})^{p^{jr}} = d^{p^{jr}} + d^{qp^{jr}} = (d - d^q) + (d - d^q)^q = 0$, contradicting the assumption $d+d^q = 1$.) Note also that if $b = (d-d^{p^{jr}})^{-1}$ then $c^{q+1} = d^{(q+1)p^{jr}}(bd-1)^{-(q+1)} = (d - d^{p^{jr}})^{q+1}$. Thus $\zeta_{A, j} \in G_{L, N}$ with $A$ given by (\ref{eq:form1}) fixes $\infty$ if and only if it is as described in (iv). 

The image of $N$ under $\zeta_{A, j} \in G_L$ with $A$ given by (\ref{eq:form2}) has equation $(ab^{q+1}+d^{p^{jr}})y = b^{q+1}z$. Hence $N$ is fixed by such an element if and only if $b^{q+1} = d(ab^{q+1}+d^{p^{jr}})$, or equivalently $a = d^{-1} - b^{-(q+1)} d^{p^{jr}}$. Plugging this into $a+a^q = 0$ and using $d+d^q=1$, we obtain $b^{q+1} = d^{q+1}$ and consequently $a = d^{-1} - d^{-(q+1)+p^{jr}}$. Thus $\zeta_{A, j}$ fixes $N$ if and only if $A$ is as given in (ii). This element of $G_{L, N}$ maps $\infty$ to $(0, 0, b^{-q})$ and so cannot fix $\infty$. 

The image of $N$ under $\zeta_{A, j} \in G_L$ with $A$ given by (\ref{eq:form3}) has equation $(ab^{q+1}d^{p^{jr}}+1)y = b^{q+1}d^{p^{jr}}z$. Hence $N$ is fixed by such an element if and only if $b^{q+1}d^{p^{jr}} = d(ab^{q+1}d^{p^{jr}}+1)$, or equivalently $a = d^{-1} - b^{-(q+1)} d^{-p^{jr}}$. Plugging this into $a+a^q = 0$ and using $d+d^q=1$, we obtain $b^{q+1} = d^{-(q+1)(p^{jr}-1)}$ and so $a = d^{-1} - d^{-(q+1)+qp^{jr}}$. Hence this $\zeta_{A, j}$ fixes $N$ if and only if $A$ is as given in (iii). Such an element of $G_{L, N}$ maps $\infty$ to $(0, b, ab)$ and hence it fixes $\infty$ if and only if $a = 0$, that is, $d^{p^{jr}} = d$. Note that if $d^{p^{jr}} = d$ then $b^{q+1} = 1$ and so $A$ is as described in (v). 

So far we have completed the proof of Claims 1.1 and 1.2.  

\delete{
(d) Based on $G_{\infty, L, N}$ above, one can easily obtain that if $\l = 1$ then $G_{\infty, L, N, \t} = \{x' = x^{p^{jr}}, y' = y^{p^{jr}} + (d-d^{p^{jr}}) z^{p^{jr}}, z' = z^{p^{jr}}: 0 \le j < 2e/r\}$. Hence $|G_{\infty, L, N, \tau}| = 2e/r$ in this case.}

\smallskip
{\bf Case 2:}~$\l \ne 1$. 
\smallskip

Denote $\mu = 1 - \l$ so that $\mu \ne 0$ and $\mu \ne d$ by our assumption. 

Consider a typical element $\zeta_{A, j}$ of $G_{L}$ with $A$ given by (\ref{eq:form1}). Set $f = (a-b)^{-1}c^{-q}$ so that $f \ne 0$ and $a = b+c^{-q}f^{-1}$. Since $b+b^q = 0$, one see that $a+a^q = 0$ if and only if $c f^q + c^q f = 0$. Thus (\ref{eq:form1}) can be rewritten as
\begin{equation}
\label{eq:form1a}
\bmat{1 & 0 & 0\\0 & f & c\\ 0 & bf + c^{-q} & bc}, 
\end{equation}
where $b \in \FFF_{q^2}, c, f \in \FFF^*_{q^2}$ such that $b+b^q=0$ and $c f^q + c^q f = 0$. 

\smallskip
{\sc Claim 2.1:}~ 
$G_{L, N}$ consists of those $\zeta_{A, j}$ (where $0 \le j < 2e/r$) such that one of the following holds:
\begin{itemize}
\item[(i)] $A$ is given by (\ref{eq:form1a}) with $c = (\mu - d)^{p^{jr}} \mu^{-p^{jr}+1} [b(\mu-d)+1]^{-1}$ and $f = [c^{q+1}-c \mu^{p^{jr}-1} (\mu - d)]c^{-q}(\mu - d)^{-p^{jr}}$, where $b \in \FFF_{q^2}$ satisfies $b+b^q=0$, $b \ne -(\mu-d)^{-1}$, $b \ne (\mu -d)^{p^{jr}-(q+1)}\mu^{-(q+1)(p^{jr}-1)} - (\mu-d)^{-1}$, (note that $b \ne -(\mu-d)^{-1}$ is satisfied automatically when $\l + \l^q \ne 1$), and in addition $j$ satisfies  
\begin{equation}
\label{eq:lambda}
\left(\frac{1-\l}{\l}\right)^{(q+1)(p^{jr}-1)} = 1
\end{equation}
when $\l + \l^q \ne 1$;
\item[(ii)] $A$ is given by (\ref{eq:form2}) with $a = (\mu -d)^{p^{jr}-(q+1)}\mu^{-(q+1)(p^{jr}-1)} - (\mu-d)^{-1}$ and $b = \mu^{q(p^{jr}-1)}(\mu-d)^q$, and in addition $j$ satisfies (\ref{eq:lambda}) when $\l + \l^q \ne 1$; 
\item[(iii)] $A$ is given by (\ref{eq:form3}) with $a = (\mu -d)^{qp^{jr}-(q+1)}\mu^{-(q+1)(p^{jr}-1)} - (\mu-d)^{-1}$ and $b = (\mu(\mu-d)^{-1})^{q(p^{jr}-1)}$, and in addition  $j$ satisfies (\ref{eq:lambda}) when $\l + \l^q \ne 1$.
\end{itemize}
\smallskip 

{\sc Claim 2.2:}~ 
$G_{\infty, L, N}$ consists of those $\zeta_{A, j}$ (where $0 \le j < 2e/r$) such that one of the following holds:
\begin{itemize}
\item[(iv)] $A$ is given by (\ref{eq:form1a}) and $j$ is such that
$\mu^{(q+1)(p^{jr}-1)} \ne (\mu-d)^{p^{jr}-1}$,
where $b = \mu^{(q+1)(p^{jr}-1)}$ $[\mu^{(q+1)(p^{jr}-1)}(\mu-d)^q - (\mu - d)^{qp^{jr}}]^{-1}$, $c$ and $f$ are as in (i) above, and in addition $j$ satisfies (\ref{eq:lambda}) when $\l + \l^q \ne 1$; or
\item[(v)] $A$ is given by (\ref{eq:form3}) and $j$ satisfies $\mu^{(q+1)(p^{jr}-1)} = (\mu-d)^{p^{jr}-1}$, where $a = 0$, $b = (\mu(\mu-d)^{-1})^{q(p^{jr}-1)}$, and in addition $j$ satisfies (\ref{eq:lambda}) when $\l + \l^q \ne 1$. 
\end{itemize}
\smallskip 

One can see that the number of integers $0 \le j < 2e/r$ satisfying (\ref{eq:lambda}) is equal to $2e/l^* r$. 
Thus, by the claims above, if $\l + \l^q \ne 1$, then $|G_{L, N}| = 2(q+1)e/l^* r$ and $|G_{\infty, L, N}| = 2e/l^* r$; and if $\l + \l^q = 1$, then $|G_{L, N}| = 2qe/r$ and $|G_{\infty, L, N}| = 2e/r$.

\smallskip
\textit{Proof of Claims 2.1 and 2.2:}~The image of $N$ under $\zeta_{A, j} \in G_{L}$ with $A$ given by (\ref{eq:form1a}) has equation $c\mu^{p^{jr}}x+\{(\mu-d)^{p^{jr}}-bc^{q}[c-f(\mu-d)^{p^{jr}}]\}y+c^{q}[c-f(\mu-d)^{p^{jr}}]z = 0$. So $N$ is fixed by this element $\zeta_{A, j}$ if and only if 
\bea
bc^{q}[c-f(\mu-d)^{p^{jr}}] & \ne & (\mu-d)^{p^{jr}} \non \\
c\mu^{p^{jr}} & = & \mu \{(\mu-d)^{p^{jr}}-bc^{q}[c-f(\mu-d)^{p^{jr}}]\} \label{eq:1} \\ 
c^{q}[c-f(\mu-d)^{p^{jr}}] & = & (\mu - d) \{(\mu-d)^{p^{jr}}-bc^{q}[c-f(\mu-d)^{p^{jr}}]\}. \label{eq:2}
\eea
These hold if and only if 
\bea
b(\mu-d)+1 & \ne & 0 \label{eq:b} \\ 
f & = & [c^{q+1}-c \mu^{p^{jr}-1} (\mu - d)]c^{-q}(\mu - d)^{-p^{jr}} \label{eq:f} \\
c & = & (\mu - d)^{p^{jr}} \mu^{-p^{jr}+1} [b(\mu-d)+1]^{-1}. \label{eq:c}
\eea 

Given $c$ and $f$ above, by using $b + b^q = 0$, one can check that $c f^q + c^q f = 0$ holds if and only if 
\begin{equation}
\label{eq:cf}
\mu^{(q+1)(p^{jr}-1)} [(\mu-d)+(\mu-d)^q] = [(\mu-d)+(\mu-d)^q]^{p^{jr}}.
\end{equation}
Note that $(\mu-d)+(\mu-d)^q = 1 - \l - \l^q$ since $d+d^q=1$. If $\l + \l^q \ne 1$, then (\ref{eq:cf}) can be rewritten as 
$(1-\l)^{(q+1)(p^{jr}-1)} = (1-\l-\l^q)^{p^{jr}-1}$, which holds if and only if $\l^{(q+1)(p^{jr}-1)} = (1-\l-\l^q)^{p^{jr}-1}$. Using this one can further verify that in this case (\ref{eq:cf}) holds if and only if $j$ satisfies (\ref{eq:lambda}). Thus, if $\l + \l^q = 1$ then $c f^q + c^q f = 0$ holds for any $j$, and if $\l + \l^q \ne 1$ then $c f^q + c^q f = 0$ if and only if $j$ satisfies (\ref{eq:lambda}). 

Note that (\ref{eq:b}) holds automatically if $\l + \l^q \ne 1$. If $\l + \l^q = 1$ then $b_1 := -(\mu-d)^{-1}$ is a solution to $x+x^q = 0$ and (\ref{eq:b}) amounts to $b \ne b_1$. With $c$ and $f$ above we have: $f \ne 0$ $\Leftrightarrow$ $c^q \ne \mu^{p^{jr}-1} (\mu - d)$ $\Leftrightarrow$ $c \ne \mu^{q(p^{jr}-1)} (\mu - d)^q$ $\Leftrightarrow$ $(\mu - d)^{p^{jr}} \mu^{-p^{jr}+1} [b(\mu-d)+1]^{-1} \ne \mu^{q(p^{jr}-1)} (\mu - d)^q$ 
$\Leftrightarrow$ $b \ne b_2 := (\mu -d)^{p^{jr}-(q+1)}\mu^{-(q+1)(p^{jr}-1)} - (\mu-d)^{-1}$. One can verify that $b_2+b_2^q = 0$ if (\ref{eq:cf}) is satisfied. So $b$ can be any solution to $x+x^q = 0$ except $b_1$ and $b_2$ when $\l + \l^q = 1$ and $b_2$ when $\l + \l^q \ne 1$. 

The above shows that an element $\zeta_{A, j} \in G_{L}$ with $A$ given by (\ref{eq:form1a}) fixes $N$ if and only if $c$ and $f$ are given by (\ref{eq:c}) and (\ref{eq:f}), respectively, where $b \in \FFF_{q^2}$ satisfies $b+b^q=0$, $b \ne b_1, b_2$, and in addition $j$ satisfies (\ref{eq:lambda}) when $\l + \l^q \ne 1$. This element maps $\infty$ to $\la 0, f, bf+c^{-q} \ra$ and hence it fixes $\infty$ if and only if $bf+c^{-q} = 0$. Using (\ref{eq:f}), (\ref{eq:c}) and $b+b^q=0$, this condition is equivalent to that 
$b [\mu^{(q+1)(p^{jr}-1)}(\mu-d)^q - (\mu - d)^{qp^{jr}}] = \mu^{(q+1)(p^{jr}-1)}$, which can not be satisfied when $\mu^{(q+1)(p^{jr}-1)} = (\mu-d)^{p^{jr}-1}$ as $\mu \ne 0$. On the other hand, if $\mu^{(q+1)(p^{jr}-1)} \ne (\mu-d)^{p^{jr}-1}$, then this condition amounts to $b = b_3 := \mu^{(q+1)(p^{jr}-1)}[\mu^{(q+1)(p^{jr}-1)}(\mu-d)^q - (\mu - d)^{qp^{jr}}]^{-1}$. 
One can check that $b_3 + b_3^q = 0$ and $b_3 \ne b_1, b_2$. 

The image of $N$ under $\zeta_{A, j} \in G_{L}$ with $A$ given by (\ref{eq:form2}) has equation $b\mu^{p^{jr}}x + [-ab^{q+1}+(\mu-d)^{p^{jr}}]y+b^{q+1}z = 0$. So $N$ is fixed by this element if and only if $a \ne b^{-(q+1)}(\mu-d)^{p^{jr}}$, 
$b\mu^{p^{jr}} = \mu [-ab^{q+1}+(\mu-d)^{p^{jr}}]$ and
$b^{q+1} = (\mu - d) [-ab^{q+1}+(\mu-d)^{p^{jr}}]$.
From this and noting $a + a^q = 0$ one can verify that such an element $\zeta_{A, j}$ fixes $N$ if and only if $j$ satisfies (\ref{eq:cf}) and $a$ and $b$ are as in (ii) of Claim 2.1. (The proof above shows that (\ref{eq:cf}) is satisfied when $\l + \l^q = 1$ and is equivalent to (\ref{eq:lambda}) when $\l + \l^q \ne 1$.) Such an element maps $\infty$ to $\la 0, 0, b^{-q} \ra$ and so it never fixes $\infty$. 

The image of $N$ under $\zeta_{A, j} \in G_{L}$ with $A$ given by (\ref{eq:form3}) has equation $b\mu^{p^{jr}}x + [1-ab^{q+1}(\mu-d)^{p^{jr}}]y+b^{q+1}(\mu-d)^{p^{jr}}z = 0$. Hence $N$ is fixed by this element if and only if $a \ne b^{-(q+1)}(\mu-d)^{-p^{jr}}$, $b\mu^{p^{jr}} = \mu [1-ab^{q+1}(\mu-d)^{p^{jr}}]$ and $b^{q+1}(\mu-d)^{p^{jr}} = (\mu - d) [1-ab^{q+1}(\mu-d)^{p^{jr}}]$. From this and using $a + a^q = 0$ one can verify that such an element $\zeta_{A, j}$ fixes $N$ if and only if $j$ satisfies (\ref{eq:cf}) and $a$ and $b$ are as in (iii) of Claim 2.1. Since this element of $G_{L, N}$ maps $\infty$ to $\la 0, b, ab \ra$, it fixes $\infty$ if and only if $a = 0$, which holds if and only if $\mu^{(q+1)(p^{jr}-1)} = (\mu-d)^{p^{jr}-1}$.  

Up to now we have completed the proof of Claims 2.1 and 2.2 and hence the proof of the lemma.
\qed
\end{proof}

\begin{proof} \textbf{of Theorem \ref{thm:2ndquotient}}~
It is clear that $\CC$ is a $G$-invariant partition of $V(q)$. 
By Lemma \ref{lem:1edge} and using the notation there, $(\infty, L)$ and $(\tau, N)$ are adjacent in $\Ga = \Ga_{r,\l}(q)$. So by the orbit-stabilizer lemma the valency of $\Ga_{\CC}$ is given by $b(\CC):=|G_L|/|G_{L,N}|$, which together with Lemma \ref{lem:2linestab} yields (a). The number of vertices in $C(L)$ having neighbours in $C(N)$ is given by $k(\CC) := |G_{L,N}|/|G_{\infty, L, N}|$. 
This together with Lemma \ref{lem:2linestab} and the $G$-symmetry of $\Ga_{\CC}$ implies that $\Ga$ is a multicover of $\Ga_{\CC}$ if $\l = 1$ or $\l \ne 1$ and $\l + \l^q \ne 1$, and $k(\CC) = q$ if $\l + \l^q = 1$. Since $\CC$ has block size $q+1$, the number of blocks of $\CC$ containing at least one neighbour of $(\infty, L)$ is given by $b(\CC)k(\CC)/(q+1)$ and so (b) follows. Since the valency of $\Ga$ is $k^*q(q^2-1)$ by Theorem \ref{thm:char} and $k^* = 1$ when $\l=1$, for adjacent $C(L_1), C(L_2) \in \CC$, by (b) the valency of the bipartite subgraph of $\Ga$ induced on $C(L_1) \cup C(L_2)$ (excluding the isolates) is equal to $q+1$ if $\l = 1$, $k^*/l^*$ if $\l \ne 1$ and $\l + \l^q \ne 1$, and $k^*$ if $\l + \l^q = 1$. In particular, if $\l=1$, then $\Ga$ is isomorphic to the lexicographic product of $\Ga_{\CC}$ and the empty graph of $q+1$ vertices. 
\qed
\end{proof} 

\begin{remark}
\label{rem:1}
{\em (a) Using the notation in Lemma \ref{lem:1edge}, if $\l + \l^q = 1$, then $\la \bu \ra = \la 0, 1-\l-d, -1 \ra$ is a point of $U_H(q)$ which is in both $L$ and $N$. Since $(\la \bu \ra, L)$ and $(\la \bu \ra, N)$ are not adjacent in $\Ga_{r,\l}(q)$, by part (c) of Theorem \ref{thm:2ndquotient}, $(\la \bu \ra, L)$ ($(\la \bu \ra, N)$, respectively) is the only vertex in $C(L)$ ($C(N)$, respectively) without neighbour in $C(N)$ ($C(L)$, respectively). 

(b) In the special case when $\la \psi^r \ra_{\l} = \la \psi^r \ra_{\eta}$, we have $k^* = l^*$ and so $\Ga_{r,\l}(q)$ is a cover of $\Ga_{r,\l}(q)_{\CC}$ by part (c) of Theorem \ref{thm:2ndquotient}.}
\end{remark}

\delete{
\SM{It would be interesting to discuss various combinatorial and group-theoretic properties of unitary graphs, e.g. when a unitary graph is connected, diameter of the unitary graphs, when two unitary graphs are isomorphic, full automorphism groups of the unitary graphs, etc.}
}

\section{Proof of Theorem \ref{thm:class}}
\label{sec:class'n}


Suppose $\Ga, G, \BB$ and $\DD(\Ga, \BB)$ satisfy the conditions in Theorem \ref{thm:class}. 
By Lemma \ref{lem:flag graph}, $\Ga$ is isomorphic to a $G$-flag graph $\Ga(\DD, \Om, \Psi)$ for a $G$-doubly point-transitive linear space $\DD$. Since such a linear space is necessarily $G$-flag-transitive, $\Om$ must be the set of all flags of $\DD$, and $\Psi$ is a self-paired $G$-orbital of $\Om$ compatible with $\Om$. Since $\Om$ clearly satisfies (A2) and (A3), $\Om$ is feasible if and only if every point of $\DD$ is incident with at least three lines and, for distinct points $\s$ and $\t$, $G_{\s\t}$ is transitive on the lines incident with $\s$ but not $\t$.

Since $\DD$ is nontrivial and $G$ is almost simple with socle $N$, by \cite[Theorem 1]{Kantor85}, $\DD$ and $(G, N)$ are as in one of the following cases: 
\begin{itemize}
\item[(i)] $\DD = \PG(d-1,q)$, $N = \PSL(d, q)$, $d \ge 3$; 
\item[(ii)] $\DD$ is the Hermitian unital $U_{H}(q)$, $N = \PSU(3, q)$, $q > 2$ a prime power;
\item[(iii)] $\DD$ is the Ree unital $U_{R}(q)$ (see \cite{Lu}), $N =$ $^2G_2(q)$ is the Ree group, $q = 3^{2s+1} \ge 3$; 
\item[(iv)] $\DD=\PG(3,2)$, $N = A_7$. 
\end{itemize}
Note that in each case above, except $N =$ $^2G_2(3)$ ($\cong \PGammaL(2, 8)$), $N$ is also doubly transitive on the points of $\DD$.
 
\delete{
\begin{itemize}
\item[(B)] $G$ has a regular normal subgroup which is elementary abelian of order $v := p^d$, where $d \ge 1$ and $p$ is a prime. Identify $G$ with a group of affine transformations $\bv \mapsto \bv^g + \bc$ of $V(d, p)$, where $\bv, \bc \in V(d, p)$ and $g$ is in the stabilizer $G_{\bf 0}$ of ${\bf 0} \in V(d, p)$ in $G$. Then one of the following occurs:
\begin{itemize}
\item[(v)] $G \le \AGammaL(1, v)$, $\DD$ is an affine space (\SM{need detail information about $\DD$});
\item[(vi)] $\DD = \AG(n,q)$, $\SL(n, q) \unlhd G_{\bf 0}$, $v = q^n$; 
\item[(vii)] $\DD = \AG(n,q)$, $\Sp(n, q) \unlhd G_{\bf 0}$, $n \ge 4$, $v = q^n$; 
\item[(viii)] $\DD = \AG(2,p)$, $G_{\bf 0} \unrhd \SL(2,3)$ or $\SL(2,5)$, $v=p^2$, $p = 5, 7, 11, 19, 23, 29$, or $59$; 
\item[(ix)] $\DD = \AG(4,3)$, $G_{\bf 0}$ contains a normal extraspecial subgroup $E$ of order $2^5$, $v=3^4$;
\item[(x)] $\DD$ is the `exceptional nearfield plan' of order $9$ having $3^4$ points and $3^2 \cdot (3^2 + 1)$ lines, with $3^2$ points in each line and $3^2 + 1$ lines through each point (see \cite{Foulser} and \cite[pp.230]{Dembowski}), $G_{\bf 0}$ contains a normal extraspecial subgroup $E$ of order $2^5$, $v=3^4$;
\item[(xi)] $\SL(2, 5) \unlhd G_{\bf 0}$, $\DD$ is $\AG(2,9)$, the exceptional nearfield plan as above or $\AG(4,3)$, $v=3^4$;
\item[(xii)] $\DD = \AG(6,3)$, $G_{\bf 0} = \SL(2, 13)$, $v=3^6$;
\item[(xiii)] $\DD$ is the Hering affine plane of order $27$ having $3^6$ points and $3^3 \cdot (3^3 + 1)$ lines, with $3^3$ points in each line and $3^3 + 1$ lines through each point (see \cite{Hering69} and \cite[pp.236]{Dembowski}), $G_{\bf 0} = \SL(2, 13)$, $v=3^6$; 
\item[(xiv)] $\DD$ is one of the two Hering designs \cite{Hering} of order $90$ having $3^6$ points and $81 \cdot 91$ lines, with $9$ points in each line and $91$ lines through each point, $G_{\bf 0} = \SL(2, 13)$, $v=3^6$. 
\end{itemize}
}

\medskip
{\bf Case (i)}~In this case we have $\PSL(d, q) \unlhd G \le \PGammaL(d, q)$ and $\Ga$ is isomorphic to $\Ga^+(P; d, q)$ or $\Ga^{\simeq}(P; d, q)$ by \cite[Theorem 3.6]{Zhou-EJC}.  

\medskip
{\bf Case (ii)}~In this case we have 
$\PGU(3, q) \unlhd G \le \PGammaU(3, q)$ and $\Ga$ is a unitary graph by Theorem \ref{thm:char}. Note that, if $3$ divides $q+1$, then $G \ne \PSL(3, q)$ by Lemma \ref{lem:fbty}(b).

\medskip
{\bf Case (iii)}~It is well known that $N :=$ $^2G_2(q)$ is transitive on the flags of $U_{R}(q)$ and $\Aut(U_{R}(q)) \cong \Aut(N) \cong$ $N \rtimes C_{2s+1}$ (see e.g.~\cite{Goren}), where $C_{2s+1}$ is the cyclic group of order $2s+1$. Identify $\Aut(N)$ with $K := N \rtimes C_{2s+1}$ and let $N \unlhd G \le K$. Let $\s$ and $\t$ be distinct points of $U_{R}(q)$. Since $G$ and $K$ are doubly transitive on the points of $U_{R}(q)$, we have $G = NG_{\s\t}$ and $K = NK_{\s\t}$. Thus $N_{\s\t} \unlhd G_{\s\t}$, $N_{\s\t} \unlhd K_{\s\t}$, $\overline{G} := G/N \cong G_{\s\t}/N_{\s\t}$ and $\la \phi \ra \cong K/N \cong K_{\s\t}/N_{\s\t}$. It is known that $N_{\s\t}$ is isomorphic to the cyclic group of order $q-1$ (see e.g.~\cite[Section 7.7]{Dixon-Mortimer}). Since $|G/N|$ divides $|K/N| = 2s+1$, we may assume $|\overline{G}| = 2s'+1 \ge 1$ for some divisor $2s'+1$ of $2s+1$ so that  $|G_{\s\t}| = (q-1)(2s'+1)$. If (A4) is satisfied by $(G, U_{R}(q))$, that is, $G_{\s \t}$ is transitive on $\Om(\s) \setminus \{(\s, L)\}$, where $L$ is the line through $\s$ and $\t$, then $q^2 - 1$ divides $|G_{\s \t}|$, which occurs if and only if $q+1$ divides $2s'+1$. However, this cannot happen since $q+1$ is even but $2s'+1$ is odd. Thus $(G, U_{R}(q))$ does not satisfy (A4) and so the set of flags of $U_{R}(q)$ is not feasible with respect to $G$. Therefore, no $G$-symmetric graph satisfying the conditions of Lemma \ref{lem:flag graph} arises from $U_{R}(q)$.  

\delete{
Ref.: P.~Kleidman, The maximal subgroups of the Chevalley groups $G_2(q)$ with $q$ odd, the Ree group $^2G_2(q)$, and their automorphism groups, {\em J. Algebra} {\bf 117} (1) (1988), 30--71; Huppert B., Blackburn N.: Finite Groups III. Springer Verlag, Berlin (1982); L\"uneburg H.: Some remarks concerning the Ree groups of type (G2), J. Algebra 3 (1966), 256--259.
}  

\medskip
{\bf Case (iv)}~We may view $A_7$ as a subgroup of $\GL(4,2)$ ($\cong A_8$) generated by 
$$
\bmat{0&1&0&1\\0&1&1&0\\0&1&0&0\\1&0&1&1},\;\;
\bmat{1&1&1&1\\0&1&1&0\\1&1&0&0\\0&0&1&0}.
$$ 
Then $A_7$ is $2$-transitive on the vertex set $V(4,2) \setminus \{{\bf 0}\}$ of $\PG(3,2)$.    
The only possibility is $G = A_7$ because $\Aut(A_7) \cong S_7$ is not a subgroup of $\Aut(\PG(3,2))$.
A typical line of $\PG(3,2)$ is of the form $L(\s\t) := \{\s, \t, \s+\t\}$, where $\s$ and $\t$ are distinct points of $V(4,2) \setminus \{{\bf 0}\}$. Since $G_{\s\t} \cong A_4$ is transitive on $V(4,2) \setminus \{{\bf 0}, \s, \t, \s+\t\}$ (see e.g.~\cite{Atlas}), $G_{\s\t}$ is transitive on the flags of $\DD$ other than $(\s, L(\s\t))$ whose point-entry is $\s$. Consequently, the set $\Om$ of flags of $\DD$ is feasible with respect to $G$. 

We now give all self-paired $G$-orbitals $\Psi = ((\s, L), (\t, N))^G$ of $\Om$ compatible with $\Om$. 
 
Consider first the case when $L$ and $N$ have a common point for some (and hence all) $((\s, L), (\t, N)) \in \Psi$. By the $2$-transitivity of $G$ on $V(4,2) \setminus \{{\bf 0}\}$, we may fix distinct $\eta, \s \in V(4,2) \setminus \{{\bf 0}\}$ and $L = L(\s \eta)$, and take $\eta$ as the common point of $L$ and $N$. It is known that each involution of $G$ fixes exactly three points in $V(4,2) \setminus \{{\bf 0}\}$. Choose an involution $g$ of $G$ which fixes $\eta$ but not $\s$. Set $\t = \s^g$ and $N = L(\t\eta)$. Then $\Psi_1 := ((\s, L), (\t, N))^G$ is self-paired and compatible with $\Om$. Since $G_{\s \eta} \le G_{\s, L}$ and $G_{\s \eta} \cong A_4$ is transitive on $V(4,2) \setminus \{{\bf 0}, \s, \eta, \s+\eta\}$, $G_{\s, L}$ is transitive on the flags of $\DD$ with point-entry $\eta$ and up to isomorphism $\Ga(\DD, \Om, \Psi_1)$ is the unique $G$-flag graph of $\DD$ in the case when $L$ and $N$ have a common point. Moreover, since $\PG(3,2)$ is $G$-flag transitive, we see that two flags $(\s_1, L_1), (\t_1, N_1) \in \Om$ are adjacent in $\Ga_1 := \Ga(\DD, \Om, \Psi_1)$ if and only if $L_1$ and $N_1$ have a common point which is neither $\s_1$ nor $\t_1$. Since there are $7$ lines passing through each point and $L$ has two points other than $\s$, this graph has valency $2 \cdot 2 \cdot 6$, diameter 2 and girth 3.

Next we consider the case when $L$ and $N$ have no common point for any $((\s, L), (\t, N))$ $\in \Psi$. Since $\DD$ is $G$-flag-transitive and $G_{\s, L}$ is transitive on the points of $\DD$ not in $L$, in searching for such self-paired $G$-orbitals $\Psi = ((\s, L), (\t, N))^G$, we may fix $\s, \t$ and $L$ without loss of generality. For a fixed choice of $\s, \t$ and $L$, there are at most four such $\Psi$ because there are exactly four lines of $\DD$ through $\t$ which have no common points with $L$. Choosing  $\s = (1,0, 0, 0), \t = (0, 0, 1, 0)$ and $L = \{(1, 0, 0, 0), (0, 1, 0, 0), (1, 1, 0, 0)\}$, these four lines are $N_1 = \{\t, (0, 0, 0, 1), (0, 0, 1, 1)\}$, $N_2 = \{\t, (1, 0, 0, 1), (1, 0, 1, 1)\}$, $N_3 = \{\t, (0, 1, 0, 1), (0, 1, 1, 1)\}$ and $N_4 = \{\t, (1, 1, 0, 1), (1, 1, 1, 1)\}$. 
Using MAGMA \cite{BCP} we verified that $\Psi_2 := ((\s, L), (\t, N_1))^G = ((\s, L), (\t, N_2))^G$, $\Psi_3 := ((\s, L), (\t, N_3))^G$ and $\Psi_4 := ((\s, L), (\t, N_4))^G$ are all self-paired, and the corresponding graphs $\Ga_2 := \Ga(\DD, \Om, \Psi_2)$, $\Ga_3 := \Ga(\DD, \Om, \Psi_3)$ and $\Ga_4 := \Ga(\DD, \Om, \Psi_4)$ are pairwise nonisomorphic. 

Since $\PG(3,2)$ has $15$ points and $35$ lines with each line containing $3$ points and 
each point contained in $7$ lines, the graphs $\Ga_i$, $i = 1,2,3,4$, have $3 \cdot 35 = 105$ vertices. We have $|G_{\s, L}| = 24$ and $G_{\s, L, \t}$ consists of the identity together with the involution 
$\bmat{1&1&0&1\\0&1&0&0\\0&0&1&1\\0&0&0&1}$
fixing both $N_3$ and $N_4$ and swapping $N_1$ and $N_2$. Thus $\Ga_2, \Ga_3$ and $\Ga_4$ have valency $24, 12$ and $12$ respectively. Using MAGMA \cite{BCP} we obtain that $\Ga_2, \Ga_3, \Ga_4$ have (diameter, girth) = $(3, 3)$, $(3, 4)$, $(3, 3)$ respectively. 
\qed

In Case (iv) of the proof above, the pointwise stablizer $G_{(L)}$ of $L$ in $G$ is generated by 
$$
\bmat{1&0&1&1\\0&1&1&0\\0&0&1&0\\0&0&0&1},\;\;
\bmat{1&0&1&1\\0&1&0&0\\0&0&0&1\\0&0&1&1}
$$ 
and is regular on the 12 points of $\PG(3,2) \setminus L$. Hence the neighbourhood of $(\sigma, L)$ in $\Gamma_3$ ($\Gamma_4$, respectively) consists of the images of $(\tau, N_3)$ ($(\tau, N_4)$, respectively) by $G_{(L)}$; and the neighbourhood of $(\sigma, L)$ in $\Gamma_2$ consists of the union of the images of $(\tau,N_1)$ and $(\tau,N_2)$ by $G_{(L)}$.

\medskip
\noindent \textbf{Acknowledgment:}  
This work was partially supported by the Italian Ministero
dell'Istruzione, dell'Universit\`a e della Ricerca (MIUR), and by the
Gruppo Nazionale per le Strutture Algebriche, Geometriche e le loro
Applicazioni (GNSAGA) during a visit of Zhou to Universit\`a di Perugia in 2010.
Zhou was also partially supported by a Future Fellowship (FT110100629) of the Australian Research Council and a Shanghai Leading Academic Discipline Project (No. S30104).

\end{document}